\documentclass{amsart}
\usepackage{amsmath,amsfonts,amssymb,amscd,amsthm,epsf}
\usepackage{pinlabel}
\newtheorem{prop}{Proposition}[section]
\newtheorem{thm}[prop]{Theorem}
\newtheorem{cor}[prop]{Corollary}
\newtheorem{ques}[prop]{Question}

\newtheorem{prob}[prop]{Problem}
\newtheorem{lem}[prop]{Lemma}

\newtheorem{adden}[prop]{Addendum}

\theoremstyle{definition}
\newtheorem{de}[prop]{Definition}
\newtheorem{example}[prop]{Example}
\newtheorem{examples}[prop]{Examples}
\theoremstyle{remark}
\newtheorem{Remark}[prop]{Remark}             
\newtheorem{Remarks}[prop]{Remarks}             
\def\C{{\mathbb C}}
\def\CP{{\mathbb C P}}

\def\Z{{\mathbb Z}}
\def\R{{\mathbb R}}
\def\Q{{\mathbb Q}}
\def\G{{\mathcal G}}

\def\D{{\mathcal D}}
\def\S4{{\mathcal S}}

\def\inter{\mathop{\rm int}\nolimits}
\def\cl{\mathop{\rm cl}\nolimits}

\def\max{\mathop{\rm max}\nolimits}
\def\id{\mathop{\rm id}\nolimits}
\def\GL{\mathop{\rm GL}\nolimits}

\def\SO{\mathop{\rm SO}\nolimits}
\def\SL{\mathop{\rm SL}\nolimits}

\def\co{\colon\thinspace}
\def\st{\thinspace |\thinspace}

\def\blow{\mathop{\overline{\C P^2}}\nolimits}
\def\blowst{stably }
\begin{document}
\title{Group actions, corks and exotic smoothings of $\R^4$}
\author{Robert E. Gompf}
\address{The University of Texas at Austin, Mathematics Department RLM 8.100, Attn: Robert Gompf,
2515 Speedway Stop C1200, Austin, Texas 78712-1202}
\email{gompf@math.utexas.edu}
\begin{abstract} We provide the first information on  diffeotopy groups of exotic smoothings of $\R^4$: For each of uncountably many smoothings, there are uncountably many isotopy classes of self-diffeomorphisms. We realize these by various explicit group actions. There are also actions at infinity by nonfinitely generated groups, for which no nontrivial element extends over the whole manifold. In contrast, every diffeomorphism of the end of the universal $\R^4$ extends. Our techniques apply to many other open 4-manifolds, and are related to cork theory. We show that under broad hypotheses, cork twisting is equivalent (up to blowups) to twisting on an exotic $\R^4$, and give applications.
\end{abstract}
\maketitle


\section{Introduction}\label{Intro}

One of the most surprising consequences of the 1980s foundational breakthroughs in 4-manifold topology was the existence of an exotic $\R^4$---a smooth manifold homeomorphic to $\R^4$ but not diffeomorphic to it. Such manifolds were soon seen to occur with  uncountably many diffeomorphism types. This contrasted sharply with previously established topology in dimensions $n\ne 4$, where every homeomorphism to $\R^n$ could be perturbed to a diffeomorphism by a $C^0$-small isotopy. (The corresponding statement is true for all manifolds when $n<4$, and when $n>4$ it follows since the contractible space $\R^n$ allows no obstructions.) One might view this anomaly as a shortage of diffeomorphisms in dimension 4. This failure of existence raises the complementary question about classifying self-diffeomorphisms of manifolds $R$ homeomorphic to $\R^4$. There are many discrete group actions on $\R^4$ that can easily be used to construct an action on an exotic $\R^4$. However, these do not immediately provide information about the diffeotopy group $\D(R)$ of (ambient) isotopy classes of diffeomorphisms of $R$, or its orientation-preserving subgroup $\D_+(R)$. In fact, it is well known that $\D_+(\R^n)$ is the trivial group for all $n$, in spite of the plethora of group actions on $\R^n$. There appears to be no previously known information about $\D_+(R)$ for any exotic $\R^4$ (although the quotient $\D(R)/\D_+(R)$ is either trivial or $\Z_2$, and it is known that both possibilities are realized). In the present article we study this problem, using recent developments in the related theory of corks. We show that there are uncountably many choices of $R$ (both large and small) for which $\D_+(R)$ is uncountable. We also see that the diffeotopy group at infinity of these manifolds is uncountable, and has a nonfinitely generated subgroup whose nontrivial elements cannot extend to diffeomorphisms of $R$. Our techniques generalize to a much larger class of open 4-manifolds. We also study the relation between these phenomena and corks, and the extent to which the diffeomorphism type of a closed 4-manifold can be varied by cutting out an exotic $\R^4$ and regluing by diffeomorphisms at infinity.

To construct an exotic $\R^4$ with nontrivial diffeotopy, we begin with a $G$-action on $\R^4$, and extend it to a $G$-action on $R$, a suitably constructed exotic $\R^4$. We then use a trick from cork theory to prove that the induced homomorphism $G\to\D(R)$ is injective. We can choose $G$ from a wide collection of groups. For example, we can take $G$ to be a direct product of infinitely many copies of $\Z$ or $\Q$, obtaining:

\begin{thm} There is an uncountable group action on $R$ that injects into $\D_+(R)$.
\end{thm}

\noindent See Theorem~\ref{diffeotopy}. Other candidates for $G$ include all countably generated free groups as well as all countable subgroups of $\SL(4,\R)$ and of the isometry groups of Euclidean and hyperbolic 3-space. In fact, all groups satisfying a general condition act on this same manifold $R$, injecting into $\D_+(R)$. Furthermore, this $R$ can be large or small, and chosen from among uncountably many diffeomorphism types in each case. The small ones can be assumed to admit Stein structures biholomorphically embedding in $\C^2$ respecting a finite  unitary action (cf.\ Remark~\ref{unitary}). The large ones can be chosen to embed in a Stein surface. Alternatively, the small ones can be chosen to admit group actions including orientation reversals, injecting into $\D(R)$, while the Stein and large ones admit no orientation-reversing self-diffeomorphisms.

We detect nontrivial elements of $\D(R)$ by their behavior at infinity. In Section~\ref{Infty}, we define the diffeotopy group $\D^\infty(R)$ of $R$ at infinity and its orientation-preserving subgroup $\D^\infty_+(R)$, in analogy with the diffeotopy group of the boundary of a compact manifold, using germs at infinity of proper embeddings. This is subtle even for the standard $\R^4$: Proposition~\ref{Schoen} states that $\D^\infty_+(\R^4)$ is trivial if and only if the 4-dimensional smooth Schoenflies Conjecture is true, and Theorem~\ref{DR4} identifies $\D^\infty_+(\R^4)$ with the group of Schoenflies balls. There is a canonical restriction homomorphism $r\co\D(R)\to\D^\infty(R)$. Casson's original construction of a small exotic $\R^4$ \cite{C} gives examples for which $D^\infty_+(R)$ contains a $\Z_2$-subgroup that intersects the image $r(D(R))$ trivially. Nothing else was previously known about $D^\infty_+(R)$ for any exotic $\R^4$. (Casson's involution is related to some peculiar $\Z_2\oplus\Z_2$-actions \cite{menag}, but these are not known to be isotopically nontrivial.) Theorem~\ref{diffeotopy} actually shows the following:

\begin{thm} Each of the $G$-actions of the previous paragraph injects into $\D^\infty(R)$. In particular, the image of $\D(R)$ in $\D^\infty(R)$ is uncountable.
\end{thm}

The kernel and cokernel of $r$ must be countable for any exotic $\R^4$ (Theorem~\ref{cok}). Nothing else seems to be known about the kernel. Is it ever nontrivial? Theorem~\ref{homeo} shows that the cokernels for our previous examples realize the maximal cardinality:

\begin{thm}\label{introEnd} There is a nonfinitely generated subgroup $G$ of $\D^\infty(R)$ that injects into $\D^\infty(R)/r(\D(R))$.
\end{thm}

\noindent In fact, any of our previous groups satisfying one additional condition can be realized as in Theorem~\ref{introEnd} by a subgroup of $\D^\infty(R)$ conjugate to one in the image of $\D(R)$. (In particular, this image is not a normal subgroup, so the cokernel is only a set.) In counterpoint, we find an exotic $\R^4$ whose cokernel has minimal cardinality (Theorem~\ref{univ}):

\begin{thm} Let $R_U$ be the Freedman-Taylor universal $\R^4$ \cite{FT}. Then $\D^\infty(R_U)/r(\D(R_U))$ is trivial.
\end{thm}

Our techniques are a gateway to analyzing a much larger class of open 4-manifolds. While we do not generalize systematically, we hint at the possibilities (see Theorem~\ref{open}), such as:

\begin{thm} Let $Y$ be a compact topological 4-manifold with connected boundary. Suppose $Y$ homeomorphically embeds in $\# n\blow$. Then there are uncountably many smoothings of $V=\inter Y$ for which the image of $\D(V)$ in $\D^\infty(V)$ is uncountable and trivially intersects some nonfinitely generated subgroup.
\end{thm}

Our results on infinite diffeotopy groups contrast sharply with Taylor's work on isometry groups \cite{Ta2}, which we summarize in Section~\ref{Review} (along with other background on exotic smoothings of $\R^4$). Slightly generalizing a condition of Taylor, we call an exotic $\R^4$ {\em full} if it has a compact subset that cannot be smoothly embedded into its own complement, or into any homology 4-sphere. Taylor essentially shows that for every metric on a full exotic $\R^4$, the isometry group is finite. In contrast, we have:

\begin{thm} (a) Each exotic $\R^4$ in Theorem~\ref{diffeotopy} admits a complete metric whose isometry group is infinite, containing a subgroup that is not finitely generated.
\item[(b)] There is a full exotic $\R^4$ whose diffeotopy groups $\D(R)$ and $\D^\infty(R)$ are uncountable.
\end{thm}

\noindent See Theorems~\ref{isom} and~\ref{fulldiff}. We exhibit uncountably many large and small diffeomorphism types satisfying (b). We present many infinite isometric group actions for the examples of Theorem~\ref{diffeotopy}. All of these inject into the diffeotopy groups $\D(R)$ and $\D^\infty(R)$, in contrast with the many isometric group actions on $\R^4$. This raises another question:

\begin{ques} Does every isometry group of an exotic $\R^4$ inject into its diffeotopy group?
\end{ques}

To present our remaining results, we need some background on smooth manifolds homeomorphic to $\R^4$. As in \cite{DF}, we refer to these as {\em $\R^4$-homeomorphs}, to avoid exclusion of $\R^4$ (and awkward pluralizations). There are two known approaches to constructing exotic $\R^4$-homeomorphs, both originating with work of Casson \cite{C} in the 1970s. The approach most immediately relevant involves the h-Cobordism Theorem. Suppose that $X$ is a smooth, closed, simply connected 4-manifold and that $W$ is an h-cobordism of $X$ (essentially a homotopy product $I\times X$ with bottom boundary identified with $X$), but that $W$ is not diffeomorphic to $I\times X$. (Such nontrivial h-cobordisms were subsequently shown to exist by Donaldson \cite{Dpoly} in the 1980s.) Casson showed that the nontriviality of $W$ could be localized over a contractible open subset $R\subset X$, which, in light of Freedman's seminal 1981 paper \cite{F}, must be homeomorphic to $\R^4$. Nontriviality of the h-cobordism could then be used to show that $R$ is an exotic $\R^4$ (cf.\ Proposition~\ref{exotic}). In fact, for suitable $W$, this $R$ can be arbitrarily chosen from an uncountable set of diffeomorphism types (DeMichelis and Freedman~\cite{DF}, following Taubes~\cite{Tb}). Casson also showed that the other boundary component of $W$, which is homeomorphic \cite{F} but typically not diffeomorphic to $X$, is obtained from $X$ by cutting out $R$ and regluing it by an involution of the end of $R$. This involution then cannot extend smoothly over $R$, so yields the $\Z_2\subset\D_+^\infty(R)$ mentioned above.

In the 1990s, an analogous phenomenon was discovered, with $R$ replaced by a compact, contractible manifold (with boundary) that is now called a {\em cork}. The first example was discovered by Akbulut \cite{A} using Kirby calculus. Then it was shown in general (Curtis, Freedman, Hsiang, Stong \cite{CFHS} and Matveyev \cite{M}) that for any h-cobordism $W$ as above, the two boundary components were related by a {\em cork twist}, cutting out a cork and regluing it by an involution of the boundary homology sphere. (The general proof is in the spirit of Casson's argument, localizing the nontriviality over the cork.) The question was immediately raised of whether more general actions on the boundary could take the place of the $\Z_2$-action given by the involution. No progress was made until the 2016 papers \cite{T}, \cite{AKMR} and \cite{InfCork}. Most notably, Auckly, Kim, Melvin and Ruberman \cite{AKMR} found {\em $G$-corks} for any finite subgroup $G$ of $\SO(4)$, contractible submanifolds with boundary $G$-actions for which cutting and regluing by distinct elements of $G$ yields distinct diffeomorphism types. A different approach \cite{InfCork} yields $\Z$-corks. The tricks used in these papers can also be applied in the exotic $\R^4$ context. However, replacing compact boundaries by noncompact ends amplifies one trick until it detects uncountable diffeotopy and ultimately gives most of the above results. In the context of cutting and regluing closed 4-manifolds $X$, we prove:

\begin{thm} Up to blowups, $G$-slice corks and $G$-slice $\R^4$-homeomorphs can be used interchangeably.
\end{thm}

\noindent See Section~\ref{Corks}. In practice, the hypotheses are satisfied by the explicit corks arising in the literature, as well as the $\R^4$-homeomorphs arising from the h-Cobordism Theorem. Much of cork theory then adapts immediately to the exotic $\R^4$ setting (Section~\ref{Apps}).

This paper is structured as follows: The terminology and notation in Section~\ref{Infty} and the background in Section~\ref{Review} are used throughout the paper; Section~\ref{Infty} also provides context for subsequent results. Sections~\ref{Groups} (actions generating uncountable diffeotopy), \ref{Univ} (the universal $\R^4$) and \ref{Corks} (the relation with corks) are independent of each other, and Section~\ref{Apps} consists of applications of Section~\ref{Corks}. All manifolds and maps are assumed to be smooth unless otherwise indicated. Embeddings are 1-1 immersions, not necessarily proper. All manifolds are implicitly oriented, with $\overline X$ denoting the manifold obtained from $X$ by reversing its orientation. All homeomorphisms and  codimension-zero embeddings preserve orientations unless otherwise stated, the main exception being that some group actions in Section~\ref{Groups} include orientation-reversing diffeomorphisms. All groups are given the discrete topology, but act by diffeomorphisms.


\section{Group actions at infinity}\label{Infty}

We begin by defining diffeomorphisms at infinity, to serve as noncompact analogs of boundary diffeomorphisms. We will use these to define group actions and diffeotopy groups at infinity, and for the noncompact analog of cutting and pasting along boundaries. Along the way, we provide context for the rest of the paper by analyzing $\D_+^\infty(\R^4)$, and proving countability of the kernel and cokernel of the restriction homomorphism $r$ for any $\R^4$-homeomorph.

We define diffeomorphisms at infinity to be germs at infinity of diffeomorphisms. To be more precise, let $V$ be an open manifold. A {\em closed neighborhood of infinity} in $V$ is a codimension-0 submanifold $Y\subset V$ that is a closed subset whose complement has compact closure. Given manifolds $V$ and $V'$, suppose $f_i\co Y_i\to Y'_i$, $i=1,2$, are diffeomorphisms (possibly reversing orientation) between closed neighborhoods of infinity $Y_i\subset V$ and $Y'_i\subset V'$. We consider $f_1$ and $f_2$ to be equivalent if they agree outside of some compact subset of $V$ containing the complements of $Y_1$ and $Y_2$.

\begin{de} A {\em diffeomorphism at infinity} from $V$ to $V'$ is an equivalence class of such diffeomorphisms. When $V$ has a single end, we will also call this a {\em diffeomorphism of the ends} of $V$ and $V'$. A diffeomorphism at infinity {\em extends over $V$} if the equivalence class contains a diffeomorphism $V\to V'$.
\end{de}

 Diffeomorphisms at infinity compose in the obvious way, making the set of self-diffeomorphisms of $V$ at infinity into a group. We define a {\em $G$-action at infinity} to be a homomorphism of a group $G$ into this group. We can similarly define homeomorphisms and topological group actions at infinity. Two diffeomorphisms at infinity will be called {\em isotopic} if they have representatives that are properly isotopic (as embeddings into $V'$). The group $\D^\infty(V)$ of isotopy classes of self-diffeomorphisms of $V$ at infinity, with orientation-preserving subgroup $\D_+^\infty(V)$, will be called the {\em diffeotopy group of $V$ at infinity}. We obtain a homomorphism $r\co\D(V)\to\D^\infty(V)$ by sending diffeomorphisms to their corresponding equivalence classes at infinity.

Before studying extendability of diffeomorphisms at infinity over an exotic $\R^4$, we examine the case of the standard $\R^4$. This reduces to a famous open conjecture.

\begin{prop}\label{Schoen}
Every self-homeomorphism at infinity extends over $\R^4$. In the smooth category, the following are equivalent:

\item[(a)] Every self-diffeomorphism at infinity extends over $\R^4$.

\item[(b)] The smooth 4-dimensional Schoenflies Conjecture: Every embedding $S^3\to S^4$ extends to an embedding of the 4-ball $B^4$.

\item[(c)] A smoothing of $B^4$ must be diffeomorphic to the standard one if its interior is diffeomorphic to $\R^4$.
\end{prop}

\begin{proof} Not (a) $\implies$ Not (b): Let $f\co Y\to Y'$ represent a diffeomorphism at infinity that does not extend over $\R^4$, and let $B\subset\R^4$ be a ball containing the complement of $Y$. Then $f|(Y-\inter B)$ cannot extend over $\R^4$, so $f|\partial B$ violates the Schoenflies Conjecture. (Note that if an embedding $S^3\to S^4$ extends to a ball on one side, then it also extends to a ball on the other side.) Since the topological Schoenflies Theorem is known \cite{Br}, \cite{Ma}, \cite{Mo}, the same reasoning proves the first sentence of the proposition.

 Not (b) $\implies$ Not (c): Suppose $S\subset S^4$ is an embedded 3-sphere violating the Schoenflies Conjecture. Then the closed complements cannot be diffeomorphic to $B^4$. However, they must be homeomorphic to $B^4$ by the topological Schoenflies Theorem. Thus, it suffices to show that the components $R_1$ and $R_2$ of $S^4-S$ are diffeomorphic to $\R^4$. Deleting a point of $S$ from $S^4$ exibits $R_1\sqcup R_2$ as the complement of a properly embedded $\R^3$ in $\R^4$. Equivalently, we have $\R^4$ exhibited as the end sum $R_1\natural R_2$. (See Section~\ref{Review}.) By the appendix  of \cite{infR4}, no exotic $\R^4$ has an inverse under end sum, so both summands are standard. (The proof is a trick known in algebra as the Eilenberg Swindle, also used by Mazur \cite{Ma} for the Schoenflies Theorem: $R_1\approx R_1\natural(R_2\natural R_1)\natural(R_2\natural R_1)\natural\cdots\approx (R_1\natural R_2)\natural(R_1\natural R_2)\natural\cdots\approx \R^4$.)

 Not (c) $\implies$ Not (a): Let $B$ be an exotic 4-ball whose interior is diffeomorphic to $\R^4$. Then a collar of $\partial B\approx S^3$  gives an embedding $\R^4-\inter B^4\approx [1,\infty)\times \partial B\to\inter B\approx\R^4$ representing a diffeomorphism at infinity that cannot extend over $\R^4$ since the image of each $\{t\}\times \partial B$ bounds a copy of $B$.
\end{proof}

Digressing briefly, we describe $\D_+^\infty(\R^4)$ more concretely. The set of $S^4$-homeomorphs (up to orientation-preserving diffeomorphism) forms a countable (possibly trivial) abelian monoid $\S4$ under connected sum. Removing a 4-handle identifies $\S4$ with the monoid of $B^4$-homeomorphs under boundary sum. The group $\S4_0$ of invertible elements of $\S4$ corresponds to the $B^4$-homeomorphs that embed in $\R^4$, or equivalently, those whose interiors are diffeomorphic to $\R^4$.

\begin{thm}\label{DR4} The groups $\D_+^\infty(\R^4)$ and $\S4_0$ are canonically isomorphic.
\end{thm}

\begin{proof} To construct a function $\sigma\co\D_+^\infty(\R^4)\to\S4_0$, let $f\co Y\to\R^4$ represent some element of $\D_+^\infty(\R^4)$. Let $B_f\subset\R^4$ be a smooth ball containing $\R^4-Y$, and let $\sigma(f)$ be the $B^4$-homeomorph bounded by $f(\partial B_f)$. This only depends on the isotopy class of $f|\partial B_f$, so is independent of choice of representative $f$ and ball $B_f$. To show that $\sigma$ is a homomorphism, let $f$ and $g$ represent arbitrary elements of $\D_+^\infty(\R^4)$. After precomposing $f$ by a translation, we may assume $B_f$ and $B_g$ are disjoint. After a further ambient isotopy of $f$ we may assume $f$ fixes $B_g$. Letting $B_{g\circ f}$ be the ambient boundary sum of $B_g$ and $B_f$ then exhibits $\sigma(g\circ f)$ as the boundary sum  of $\sigma(g)$ and $\sigma(f)$. Now surjectivity follows as in the last paragraph of the previous proof. Triviality of the kernel follows from triviality of $\D_+(\R^4)$, which can be seen by representing a given element of $\D_+(\R^4)$ by a diffeomorphism fixing 0 and its tangent space, then conjugating by dilations.
\end{proof}

Using Proposition~\ref{Schoen}, we obtain cardinality information about restriction homomorphisms.

\begin{thm}\label{cok} For every $\R^4$-homeomorph $R$, the kernel and cokernel of $r$ are countable.
\end{thm}

\begin{proof}
Exhibit $R$ as a nested union of compact submanifolds $K_n$. Each element of the kernel of $r$ is represented by a self-diffeomorphism $f$ of $R$ that is the identity outside of some $K_n$. Since $K_n$ is a compact manifold, it has only countably many isotopy classes of self-diffeomorphisms fixing $\partial K_n$. One of these, extended by the identity over $R$, contains $f$. Ranging over countably many $n$, we obtain the entire kernel of $r$ as the image of a countable set.

For the cokernel, Proposition~\ref{Schoen} shows that every element of $\D^\infty(R)$ extends to a homeomorphism $f$ of $R$ that is a diffeomorphism in the complement of some $\inter K_n$. Then $K_n$ inherits a new smoothing pulled back by $f$. Suppose the corresponding homeomorphism $g$ for another element of $\D^{\infty}(R)$ also restricts diffeomorphically to $R-\inter K_n$. If the smoothings induced on $K_n$ by $f$ and $g$ are related by a diffeomorphism $\varphi$ rel boundary, then we extend $\varphi$ over $R$ by the identity, and see that $g\circ\varphi\circ f^{-1}$ is a diffeomorphism. Replacing $g$ by $g\circ\varphi$, we conclude that the two elements of  $\D^{\infty}(R)$ lie in the same coset of $r(\D(R))$. Since there are only countably many rel boundary diffeomorphism types of compact manifolds with boundary $\partial K_n$, we see that only countably many cosets have representatives $f$ as above for each fixed $n$. Ranging over all $n$ yields the entire cokernel.
\end{proof} 

 We conclude the section with our cut-and-paste conventions. For a codimension-0 embedding $C\subset X$ of a compact manifold and a diffeomorphism $f\co \partial C\to\partial C$, we can form a new manifold $X_f$ by cutting out $C$ and regluing it via $f$. Analogously, for an open manifold $R\subset X$ and a self-diffeomorphism $f$ of $R$ at infinity, we obtain a well-defined manifold $X_f$ by cutting out a sufficiently large compact subset of $R$ and regluing a copy of $R$ by a representative of $f$. We can alternatively modify $X$ by replacing $R$ with a different manifold that is diffeomorphic to $R$ at infinity. Analogously to the compact case, the diffeomorphism type of $X_f$ only depends on $f$ through its class in $\D^\infty(R)$. (Given isotopic embeddings $f_i\co Y\to R$, the manifolds $X_{f_i}$ are determined by their restrictions to $\partial Y$. Fixing these, we can assume the maps $f_i$ agree near infinity by the Isotopy Extension Theorem.) When $C$ is contractible, every $f$ extends homeomorphically over $C$. (For example, realize $C\cup_f C$ as the boundary of a contractible 5-manifold, identified as an h-cobordism relative to a product structure over $\partial C$ realizing $f$, and apply Freedman's topological h-Cobordism Theorem.) Thus, $X_f$ is homeomorphic to $X$. The same conclusion holds for $R$ homeomorphic to $\R^4$ by Proposition~\ref{Schoen}. The corresponding statements fail in the smooth category, giving us a useful way to recognize when some $\R^4$-homeomorphs are exotic:

\begin{prop}\label{exotic} Let $f$ be a self-diffeomorphism at infinity of an $\R^4$-homeomorph $R$. Suppose that for some embedding $R\subset X$, the manifolds $X_f$ and $X$ are distinguished by their Seiberg-Witten invariants. Then $R$ has a compact subset that cannot be surrounded by a smooth 3-sphere. In particular, $R$ is not diffeomorphic to $\R^4$.
\end{prop}

\begin{proof}
Abusing notation, let $f\co Y\to R$ represent the diffeomorphism at infinity. We construct $X_f$ by removing the compact subset $K=X-\inter Y$ from $X$ and gluing in another copy of $R$ via $f$. Suppose $K\subset R$ is surrounded by a smooth 3-sphere $S\subset Y\subset R$.  Then $S$ also lies in $X$. By the topological Schoenflies Theorem, $S$ bounds a $B^4$-homeomorph $B$ in $R$, and another such $B'$ in the reglued copy of $R$. Thus, $X$ and $X_f$ are both obtained from $X-\inter B$ by attaching a $B^4$-homeomorph. By a well-known ``neck stretching" argument, the Seiberg-Witten invariants cannot distinguish $X_f$ from $X$.
\end{proof}

\begin{Remark} By similar reasoning, the smooth Schoenflies Conjecture is equivalent to the assertion that the diffeomorphism type of a 4-manifold cannot be changed by a diffeomorphism at infinity of an embedded standard $\R^4$, or that 4-manifolds must be diffeomorphic if they are diffeomorphic after a point is removed. If these statements are false, they already fail for the 4-sphere. (An exotic 4-sphere is obtained from a counterexample to Proposition~\ref{Schoen}(c) by adding a 4-handle. It becomes standard when a point is removed, and is made from gluing two copies of the standard $\R^4$ whose images are each the complement of a point.)
\end{Remark} 


\section{Review of exotic smoothings of $\R^4$}\label{Review}

The known exotic $\R^4$-homeomorphs fall into two types, both originating in Casson's work \cite{C}. These arise from the topological success (Freedman \cite{F}) and smooth failure (Donaldson \cite{D}, \cite{Dpoly}) of high dimensional topology in dimension 4. One type stems from surgery theory, and the other from the h-Cobordism Theorem as discussed in Section~\ref{Intro}. We now review both types. A more extensive overview of the foundations appears in \cite{GS}.

Using topological surgery theory, one can construct an $\R^4$-homeomorph $R_L$ that is diffeomorphic at infinity to a simply connected, spin 4-manifold $V_L$ whose intersection form is positive definite and nontrivial. Then $R_L$ is necessarily exotic, since it cannot embed in any closed, positive definite 4-manifold $X$: If it did, we could replace $R_L\subset X$ by $V_L$, obtaining a closed 4-manifold with a definite but nondiagonalizable intersection form, contradicting Donaldson's Theorem \cite{D} (or \cite{FSpi1}, \cite{Dpi1} in the nonsimply connected case). In fact, $R_L$ contains a compact subset $K_L$ that cannot embed in such an $X$. Simply choose $K_L$ to contain a slightly smaller exotic $\R^4$ for which the above argument still applies. (Throughout the subsequent discussion, our compact subsets can be taken, after enlargement, to be smooth 4-manifolds with boundary.) Every $\R^4$-homeomorph $R$ contains a {\em radial family} of $\R^4$-homeomorphs, obtained by choosing a homeomorphism $h\co\R^4\to R$ and letting $R_r$ be the image of the open ball at 0 with radius $r$ (with its smoothing inherited as an open subset of $R$). We can assume \cite{Q} that $h$ is a local diffeomorphism near the $x_1$-axis. If $R_{r_0}$ contains the compact manifold $K_L$, then the members of the uncountable radial family $\{R_r\st r\ge r_0\}$ are exotic and pairwise nondiffeomorphic. More generally, given embeddings $K_L\subset R'\subset R$ with $R$ and $R'$ homeomorphic to $\R^4$ and the closure of $R'$ compact in $R$, the manifolds $R$ and $R'$ cannot even have diffeomorphic ends. Otherwise, one could replace a subset of $K_L$ in $R$ by $V_L$, then extend the end by consecutively attaching infinitely many copies of the region between the ends of $R'$ and $R$, to obtain a manifold with a ``periodic end" and nondiagonalizable intersection form. This would contradict an extension of Donaldson's Theorem due to Taubes \cite{Tb}. As before, one can apply the argument to a slightly smaller $\R^4$-homeomorph with compact closure in $R$, to find a compact subset of $R$ that does not embed in $R'$.

To more systematically organize $\R^4$-homeomorphs by their compact subsets, we write $R\le R^*$ if every compact subset (equivalently, compact, codimension-0 submanifold) of $R$ embeds in $R^*$, and call $R$ and $R^*$ {\em compactly equivalent} if $R\le R^*\le R$, i.e., if $R$ and $R^*$ share all the same compact submanifolds. This is an equivalence relation, with the set of compact equivalence classes forming a partially ordered set. We call an $\R^4$-homeomorph {\em small} if it is compactly equivalent to $\R^4$, and {\em large} otherwise. The above uncountable radial family  $\{R_r\st r\ge r_0\}$ of large $\R^4$-homeomorphs then injects into the set of compact equivalence classes, and its image has the order type of its parameter set $[r_0,\infty)$. (With more work, one can realize the order type of $\R^2$.) One can further measure the size of $R$ using its {\em Taylor invariant} $\gamma(R)$ \cite{Ta}. This is the smallest $b$ for which every compact subset of $R$ can be embedded into some closed, spin 4-manifold with $b_+=b_-=b$. If no such $b$ exists, we set $\gamma(R)=\infty$. The Taylor invariant is a monotonic function on the poset of compact equivalence classes. Clearly, $\gamma(R)$ vanishes if $R$ is small and is positive if $R$ contains $K_L$. We henceforth assume that $R_L$ has been arranged to have finite Taylor invariant. (If necessary, we replace $R_L$ by a slightly smaller $R_r\subset R_L$. This is embedded in some compact submanifold of $R_L$, which we then double to obtain a suitable spin manifold containing $R_r$ and hence all of its compact subsets.) We also assume that $R_L$  embeds in $\blow$, so it has no orientation-reversing self-diffeomorphism. (This can be done without sacrificing the spin condition on $V_L$; e.g.\ \cite[Example~5.10]{Ta} or \cite[Theorem~3.4]{B}.)

To obtain a small exotic $\R^4$, we apply the h-cobordism method discussed in Section~\ref{Intro}. Given a nontrivial h-cobordism $W$ of a closed, simply connected 4-manifold $X$, we obtain an exotic $\R^4$ $R\subset X$ and an involution $\tau$ of the end of $R$ such that the other boundary component of $W$ is $X_\tau$. Then $R$ is a small exotic $\R^4$ since it canonically embeds in $S^4$. In fact, the entire nontrivial part of the h-cobordism over $R$ embeds into the trivial h-cobordism $I\times S^4$, e.g.\ \cite[after Exercise~9.3.3]{GS}, showing that $S^4_\tau$ is diffeomorphic to $S^4$. We focus on the h-cobordism that generated Akbulut's original cork \cite{A} and was subsequently used to give a simple handle description of an exotic $\R^4$ \cite{BG}, \cite{GS}. Let $X=K3\#\blow$ be a blown up K3-surface. This is h-cobordant, but not diffeomorphic, to $\# 3\CP^2\#20\blow$, which is $X_\tau$ for a suitably embedded small $\R^4$ $R_S\subset X$ with involution $\tau$ of its end. Let $K_S\subset R_S$ be a compact subset whose complement lies in the domain of the given diffeomorphism $\tau$. For any radial family $\{R_r\st r\ge r_0\}$ in $R_S$ with $K_S\subset R_{r_0}$, the periodic end technology of Taubes can be applied to show that no two of the pairs $(R_r,K_S)$ are diffeomorphic (\cite{menag}, based on a similar example of DeMichelis and Freedman \cite{DF}). Since there are only countably many isotopy classes of embeddings of $K_S$ into a fixed $R_r$, it follows that each diffeomorphism type in $\{ R_r\}$ is realized only countably often (although sometimes more than once, Remark~\ref{nonunique}), so we obtain uncountably many diffeomorphism types of small $\R^4$-homeomorphs (with the cardinality of the continuum in ZFC set theory) \cite{DF}. We can construct these so that $\tau$ preserves each $R_r$, and so that uncountably many diffeomorphism types admit {\em Stein structures} \cite{Ann}, i.e., complex structures properly biholomorphically embedding in $\C^N$ for $N$ sufficiently large. By \cite{steindiff}, whenever a contractible open manifold embeds in $\R^4$ and admits a Stein structure, it admits one that biholomorphically embeds in a ball in $\C^2$ (so is a {\em domain of holomorphy}).  In contrast, an exotic $\R^4$ with nonzero Taylor invariant cannot admit a Stein structure \cite{Ta}. However, we can choose our large exotic $R_L$ to smoothly embed into some Stein surface \cite{B}. (We could alternatively arrange $K_L$ not to embed in any Stein surface \cite{B}, but this would sacrifice our crucial embedding into $\blow$.)

We can construct more $\R^4$-homeomorphs with the {\em end sum} operation, the noncompact analog of the boundary sum. (This appears to have been first introduced in \cite{3R4}, \cite{infR4}; a more comprehensive treatment will be given in \cite{CG}.) We sum two open 4-manifolds $V_1$ and $V_2$ by using $I\times \R^3$ like a piece of tape: Choose a ray in each $V_i$, i.e., a proper embedding $\gamma_i\co [a,\infty)\to V_i$, then form the identification space of $V_1\sqcup V_2$ with $I\times \R^3$, identifying $[0,\frac12)\times \R^3$ with a tubular neighborhood of the image of $\gamma_1$, and $(\frac12,1]\times \R^3$ with a neighborhood of the image of $\gamma_2$, matching 4-manifold orientations. More generally, we can sum together any countable family $\{V_s\st s\in \Sigma\}$, by summing each $V_s$ to the identity element $\R^4$ using a {\em multiray}, a proper embedding $\gamma\co \Sigma\times [a,\infty)\to\R^4$ (with the discrete topology on $\Sigma$). This sum $\natural_{s\in \Sigma}V_s$ is independent of all choices, provided (for example) that each $V_s$ is simply connected at infinity. Furthermore, it is unchanged if each ray is truncated by restricting to $[b_s,\infty)$ for some $b_s>a$, so we allow a multiray to have variable initial points in its domain. (For more on when the sum is well-defined, see \cite{CG}.) In place of a multiray, it is enough to allow a disjoint family of rays in $\R^4$ indexed by $\Sigma$ (with the union of the images not necessarily closed). Properness of the resulting combined map can then be arranged by truncating so that the $n^{\rm th}$ ray avoids the ball of radius $n$. For $\R^4$-homeomorphs, end sums are well-defined on compact equivalence classes and preserve the partial ordering: If $R_s\le R'_s$ for each $s\in \Sigma$ then $\natural_{s\in \Sigma}R_s\le\natural_{s\in \Sigma}R'_s$. (Every compact subset of $\natural_{s\in \Sigma}R_s$ can be enlarged to a boundary sum of compact submanifolds of the summands.) The Taylor invariant is subadditive: $\max_{s\in \Sigma}\gamma(R_s)\le \gamma(\natural_{s\in \Sigma}R_s)\le\sum_{s\in \Sigma}\gamma(R_s)$ \cite{Ta}. Thus, a finite sum of copies of $R_L$ has finite Taylor invariant. However, an infinite sum of these has infinite Taylor invariant: If there were an integer $b$ for which every compact submanifold of the infinite sum embedded in a closed, spin 4-manifold with $b_+=b_-=b$, then we could embed an arbitrarily large finite disjoint union of copies of $K_L$. Using these to glue in copies of $V_L$, we would obtain a spin manifold with $b_-$ fixed and $b_+$ arbitrarily large, contradicting Furuta's $\frac{10}{8}$-Theorem \cite{Fu}. Given an embedding $R\subset V$, it is not clear that $R\natural R$ embeds in $V\natural V$, since the point-set boundary of $R$ in $V$ could be complicated. (Consider an open ball in $\R^2$ minus a spiral approaching its boundary.)  However, if $R'\subset R$ lies in a radial family, then $R'\natural R'$ lies in $V\natural V$: Simply locate an arc in $V\natural V$ that intersects each copy of $R'$ in its smooth nonnegative $x_1$-axis, and add a neighborhood of it to $R'\sqcup R'$. If $V$ is Stein, it follows that every end sum of copies of $R'$ embeds in a Stein surface, since every end sum of Stein surfaces admits a Stein structure. (The end sums can be realized by incorporating Eliashberg 1-handles \cite{CE} into a handle decomposition of their disjoint union.)

We will now combine both exotic $\R^4$ constructions to obtain what we will call a {\em doubly uncountable family}: an uncountable set of compact equivalence classes, each of which contains uncountably many diffeomorphism types. Such an example was first produced in \cite{menag}, but we will need a sharper result. In the construction of $R_S$, the closed manifolds $X$ and $X_\tau$ are distinguished by their Donaldson invariants (equivalently, Seiberg-Witten invariants), so they remain nondiffeomorphic after  blowups (connected sums with copies of $\blow$). The Taubes machinery is not disturbed by blowups (provided that we reverse all orientations in the context of $R_L$ so that definiteness is preserved). Thus, we can replace $R_S$ in the construction by (for example) its end sum with $R_L\subset\blow$, by blowing up to allow the necessary embeddings. Lemma~7.3 of \cite{MinGen} suitably generalizes this method from \cite{menag} and implies the following:

\begin{lem}[\cite{MinGen}]\label{radial} Let $\{R_r\st r\ge r_0\}$ be a radial family in $R_S$ with $K_S\subset R_{r_0}$. Let $\{V_r\st r\ge r_0\}$ be any family of open 4-manifolds such that each compact subset of each $V_r$ embeds in a (finite) connected sum of copies of $\blow$. Then the family  $\{R_r\natural V_r\st r\ge r_0\}$ (for any choices of defining rays)  contains uncountably many diffeomorphism types.
\end{lem}

\noindent From this, the following lemma produces our doubly uncountable families of $\R^4$-homeomorphs. (The lemma also includes a related technical statement used in the proof of Theorem~\ref{diffeotopy}). Let $\{R^*_t\st t\ge t_0\}$ be a radial family in $R_L$ with $K_L\subset R^*_{t_0}$. Fix $n\in\{0,1,2,\dots,\infty\}$. For each $r\ge r_0$ and $t\ge t_0$, let $R_{r,t}$ be the end sum of $R_r$ with $n$ copies of $R^*_t$ and a countable collection of small $\R^4$-homeomorphs.

\begin{lem}\label{doubly} (a) The set $\{R_{r,t}\st  r\ge r_0,\ t\ge t_0\}$ contains uncountably many diffeomorphism types. If $0<n<\infty$, it realizes a doubly uncountable family.
\item{(b)} There are uncountably many diffeomorphism types of Stein $\R^4$-homeomorphs $R$ embedded in $\R^4$, admitting no orientation-reversing self-diffeomorphisms, each arising as an infinite end sum of copies of some $\R^4$-homeomorph that embeds $\tau$-equivariantly in an infinite blowup of $R_S$, covering $K_S$. These embeddings extend to an embedding of $R$ into a blown up infinite sum of copies of $R_S$.
\end{lem}

\begin{proof}
For (a), note that every compact subset of an infinite end sum lies in a finite sub-sum. For fixed $t$, Lemma~\ref{radial} then gives uncountably many diffeomorphism types, all compactly equivalent to the $n$-fold end sum $\natural nR^*_t$. For $0<n<\infty$, it is easy to realize $\{\natural nR^*_t\st  t\ge t_0\}$ as a radial family with smallest member containing $K_L$, so it injects into the poset of compact equivalence classes, with the order type of its parameter set. The proof of (b) is related but harder, so deferred to the end of Section~\ref{Corks}.
\end{proof} 

\noindent When $n=0$, each $\R^4$-homeomorph $R_{r,t}$ is small, so they are all compactly equivalent. Otherwise, they are all large, with $\gamma(R_{r,t})>0$. When $n=\infty$, there is no apparent way to distinguish compact equivalence classes, since the manifolds $\natural n R^*_t$ no longer fit into a radial family: The members of a radial family must have compact closure inside the ambient manifold, and hence, finite Taylor invariant. Lemma~7.3 of \cite{MinGen}, and hence the above lemmas, apply to a nested family $\{ R_r\}$ that is not necessarily radial, but parametrized by an uncountable subset of an interval (with $\cl R_r\subset R_{r'}$ compact and $K_S\subset R_r$ whenever $r<r'$). This observation will be useful in the proofs of Lemma~\ref{doubly}(b) and Theorem~\ref{mainThm}.

Next, we present a slightly generalized and simplified version of Taylor's work on isometries \cite{Ta2}. Taylor begins by defining a {\em barrier $S^3$} in an $\R^4$-homeomorph to be a flat, topologically embedded 3-sphere that can't be moved off of itself by a diffeomorphism (from a neighborhood of the 3-sphere to another open subset). We apply the same idea to more general compact subsets.

\begin{de} An $\R^4$-homeomorph $R$ will be called {\em full} if there is a compact subset $K\subset R$ that does not embed in $R-K$ or in any homology 4-sphere.
\end{de}

\begin{thm}[cf.\ \cite{Ta2}]\label{full} (a) An $\R^4$-homeomorph $R$ is full if (and only if) it contains compact subsets $K_1$ and $K_2$ such that $K_1$ does not embed in any homology 4-sphere and there is a finite bound on the number of copies of $K_2$ that disjointly embed in $R-K_2$.
\item{(b)} An $\R^4$-homeomorph $R$ is full if it contains a copy of $K_L$ and has finite Taylor invariant.
\item{(c)} An infinite end sum of copies of a fixed $\R^4$-homeomorph cannot be full.
\end{thm}

\noindent The proof of (b) is a simplification of Taylor's proof that barriers exist under the same hypotheses. (The statement of that theorem is missing the crucial spin hypothesis for $V_L$.)
\begin{proof}
For (a), embed the maximal number of disjoint copies of $K_2$ in $R-K_2$, and let $K\subset R$ be a compact set containing all of these copies together with $K_1$ and $K_2$. Then $K$ satisfies the definition. For (c), note that every compact subset of the infinite sum lies in a finite sub-sum and so has infinitely many disjointly embedded copies. For (b), we show that $R$ satisfies (a). Set $K_1=K_2=K_L$. This does not embed in any homology 4-sphere (or other positive definite 4-manifold). Suppose we can embed $n$ disjoint copies of $K_L$ into $R-K_L$. Since their union is compact, it also embeds into a closed, spin 4-manifold $X$ with $b_+(X)=b_-(X)\le\gamma(R)$. Gluing in $n$ copies of $V_L$, we get a closed, spin manifold with $b_+=b_+(X)+nb_+(V_L)$ and $b_-=b_-(X)\le\gamma(R)$. Since $\gamma(R)$ is finite, Furuta's $\frac{10}{8}$-Theorem gives an upper bound on $n$.
\end{proof}

\begin{thm}[Taylor \cite{Ta2}]\label{Taylor} For every $C^1$ Riemannian metric on a full $\R^4$-homeomorph, the isometry group is finite.
\end{thm}

\noindent Taylor begins by essentially proving that a compact subset that cannot be moved off of itself guarantees compactness of the Lie group of isometries. (This is stated for a barrier $S^3$, but works in general.) Then, for any exotic $\R^4$ with a smooth circle action, he shows that every compact subset can be embedded into a homology 4-sphere. Thus, fullness guarantees that the isometries form a compact, 0-dimensional Lie group.

\section{Uncountable diffeotopy}\label{Groups}

We now present and prove our main theorems on the diffeotopy groups $\D(R)$ and $\D^\infty(R)$ of a large family of $\R^4$-homeomorphs. First, we specify the groups that will act on our manifolds $R$. Let $\Sigma$ be a countable, discrete set, and let $\Gamma\co \Sigma\times [0,\infty)\to\R^4$ be an injection whose restriction $\gamma_s$ to each $\{s\}\times [0,\infty)$ is a ray. For example, we could take $\Gamma$ to be the inclusion $\Q^3\times[0,\infty)\subset\R^3\times \R=\R^4$ where $\Q^3$ is given the discrete topology. Let $G$ be any group in the discrete topology.

\begin{de} A $G$-action on $\R^4$ will be called {\em $\Gamma$-compatible} if $G$ acts effectively on $\Sigma$ so that for each $g\in G$ and $(s,t)\in \Sigma\times [0,\infty)$ we have $g\circ\Gamma(s,t)=\Gamma(g(s),t)$, and so that the stabilizer of each $s\in \Sigma$ fixes (pointwise) a neighborhood of $\gamma_s([0,\infty))$.
\end{de}

Let $\G$ denote the set of all groups $G$ such that, for some choice of $\Gamma$, $G$ acts $\Gamma$-compatibly on $\R^4$. Let $\G_+$ denote the subset for which the action can be chosen to preserve orientation, and let $\G^*\subset\G_+$ denote the subset for which there is such an action fixing a neighborhood of some ray $\gamma$ in $\R^4$ whose image is disjoint from that of $\Gamma$. Without loss of generality, we may assume that $\Sigma$ is infinite in each case: When $\Sigma$ is finite, we can equivariantly replace each $s\in \Sigma$ by a copy of $\Z^+$.

\begin{prop} The sets $\G$, $\G_+$ and $\G^*$ are closed under passing to subgroups. The set $\G^*$ is closed under  countable direct products.
\end{prop}

\begin{proof}
The first sentence is obvious. For the second, take any $G\in \G^*$ with its action, $\Gamma$ and $\gamma$. Identify the countable set $\Sigma$ with $\Z^+$. Inside the $G$-fixed neighborhood of $\gamma([0,\infty))$, construct a closed tubular neighborhood $C$ such that the fibers over $[n-1,n]$ are disjoint from the first $n$ rays of $\Gamma$. Then each $\gamma_n^{-1}(C)$ is compact, so each $\gamma_n$ continues to be proper if we use $\inter C$ to perform an end sum on $\R^4$. Now given a countable family of such $\Gamma_i$-compatible groups $G_i$ in $\G^*$, we can end sum all of their corresponding $G_i$-spaces to another copy of $\R^4$ (with the trivial action) along a multiray $\Gamma'$ in the latter. The resulting manifold is diffeomorphic to $\R^4$ with a $(\prod G_i)$-action that is $\Gamma$-compatible, where $\Gamma$ is made by combining all of the maps $\Gamma_i$.
\end{proof}

\begin{examples}\label{groups} (a) The group $\Q^3$ lies in $\G^*$. (Take $\Gamma$ as preceding the definition, and the obvious action on $\R^3\times[0,\infty)$, tapered to be trivial on $\R^3\times(-\infty,-1]$.) By the proposition, we now have many uncountable groups such as $\Q^\omega$ in $\G^*\subset \G_+\subset\G$. We can similarly realize uncountable nonabelian groups in $\G$ by by including an additional $\Gamma'$-compatible action on the last copy of $\R^4$ in the previous proof (where the groups $G_i$ over each of its orbits are isomorphic).

\item[(b)] Every countably generated free group lies in $\G_+$. Simply take a ray in an end sum of copies of $S^1\times\R^3$, then pass to the universal cover with its action by deck transformations. One can similarly get an action with orientation reversals by including a copy of the twisted $\R^3$-bundle over $S^1$.

\item[(c)] Every countable group $G$ of Euclidean or hyperbolic isometries of $\R^3$ lies in  $\G$ or (if preserving orientation) $\G_+$. To see this, choose a point $p\in\R^3$ with trivial stabilizer (by avoiding countably many subsets of positive codimension). Define $\Gamma\co G\times [0,\infty)\to\R^3\times \R=\R^4$ by $\Gamma(g,t)=(g(p),t)$. Similarly, every countable subgroup of $\SL(4,\R)$ lies in $\G_+$ and every countable subgroup of $\GL(4,\R)$ whose elements have determinant $\pm1$ lies in $\G$. (This time, $\Gamma$ is composed of radial rays. The determinant condition rules out positive multiples of the identity, so that each nontrivial group element changes the direction of generic vectors.)
\end{examples}

\begin{thm}\label{diffeotopy} Each of the following three descriptions is satisfied by uncountably many diffeomorphism types of $\R^4$-homeomorphs $R$.

\item[(a)] $R$ embeds in $\R^4$. Every $G\in\G$ has an action on $R$ that injects into the diffeotopy groups $\D(R)$ and $\D^\infty(R)$, with elements of $G$ preserving orientation on $R$ if and only if they do so in the defining  $G$-action on $\R^4$.

\item[(b)] $R$ admits a Stein structure embedding biholomorphically in a ball in $\C^2$, admits no orientation-reversing self-diffeomorphism, and satisfies the previous condition for each $G\in\G_+$.

\item[(c)]  $R$ is large, embeds in a Stein surface, admits no orientation-reversing self-diffeomorphism, and satisfies the condition in (a) for each $G\in\G_+$.
\end{thm}

\begin{cor} There are uncountably many large $\R^4$-homeomorphs and uncountably many small Stein $\R^4$-homeomorphs whose diffeotopy groups $\D(R)$ and $\D^\infty(R)$ are uncountable. \qed
\end{cor}

\begin{thm}\label{homeo} Each $R$ constructed in Theorem~\ref{diffeotopy} has a self-homeomorphism $h$ with the following properties:

\item[(a)]  $h$ restricts to an orientation-preserving diffeomorphism at infinity, and

\item[(b)]  for each $G$ in the theorem, if some element of $\Sigma$ has trivial stabilizer then the $G$-action on $R$ can be chosen so that after conjugation by $h$, the resulting smooth action at infinity injects into $\D^\infty(R)/r(\D(R))$.
\end{thm}

Since $h$ preserves orientation, it is topologically ambiently isotopic to the identity \cite{Q}. Thus, the conjugated group action at infinity is both diffeomorphic at infinity to the original action on $R$ and topologically (but not smoothly) isotopic to it. Since the injection $G\to\D^\infty(R)$ was changed by conjugation in the group, the subgroup $r(\D(R))$ is not normal, and the cokernel of $r$ is only a set. Note that all of our previous examples of groups in $\G$ satisfy the extra condition in (b), except for those involving direct products.

\begin{proof}[Proof of Theorem~\ref{diffeotopy}] We construct our manifolds $R$ as end sums. For a fixed $G$ and $\Gamma$, identify $\Sigma$ with $\Z^+$. We wish to find an increasing sequence $(t_n)$ such that when each ray $\gamma_n$ is restricted to $[t_n,\infty)$, we obtain a (proper) multiray with an equivariant tubular neighborhood map, whose restrictions over each $\gamma_n$ we denote by $\varphi_n\co[t_n,\infty)\times\R^3\to \R^4$. (Here, equivariant means $g\circ\varphi_n(x)=\varphi_{g(n)}(x)$ whenever both sides are defined.) We construct the sequence by induction on $n$, with $t_1=0$: Given $n\ge1$, suppose that for each $i\le n-1$ we have already equivariantly defined each $\varphi_i|[t_i,t_n]\times\R^3$ so that its image lies in a compact set $K_{in}$ disjoint from each other $K_{jn}$, and from each $\gamma_j((t_n,\infty))$ with $j\le n$. Choose $t_{n+1}\ge t_n+1$ so that $\gamma_{n+1}([t_{n+1},\infty))$ avoids each $K_{in}$ and the ball of radius $n$. Then define $\varphi_i|[t_n,t_{n+1}]$ for each $i\le n$ so that the induction hypotheses are extended to $n+1$, completing the induction. Now the conclusions of the theorem follow easily from Section~\ref{Review}, except for injection into the diffeotopy groups: If we end sum $\R^4$ with a copy of a fixed exotic $\R^4$ along each $\varphi_n$, the resulting exotic $\R^4$ inherits a $G$-action. This sum $R$ is independent of choice of $G\in\G_+$ since infinite end sums respecting orientation are well-defined. Using $R_S$ from  Section~\ref{Review} as the model exotic summand gives the Stein condition of (b). (Shrink the model $R_S$ slightly if necessary to allow ambient end sums.) For (a), the only complication is that $G$ need not preserve orientation, so when it does preserve orientation we take our model summand to be $R_S\natural\overline R_S$ to keep the manifold independent of $G$. For (c), we use $R_S\natural R_L$ and note that there is no orientation-reversing diffeomorphism, since every compact subset of $R$ embeds in some $\#n\blow$ but $\overline K_L$ does not. In each case, we can choose $R$ from among uncountably many diffeomorphism types by equivariantly varying the model summand as in Lemma~\ref{doubly}, with (b) of that lemma preventing orientation-reversing diffeomorphisms in the Stein case.

To show that $G$ injects into $\D^\infty(R)$ in each case, suppose that some nontrivial $g\in G$ maps to the identity. Since $G$ acts $\Gamma$-compatibly, it acts effectively on $\Sigma$. Choose $s\in \Sigma$ with $g(s)\ne s$, and identify the corresponding copy of $R_S$ with the originally defined $R_S\subset X$ from the h-cobordism construction. In Case (a), every other summand of $R$ embeds in the identity element $\R^4$, so we obtain an embedding $R\subset R_S\natural_\infty\R^4\approx R_S\subset X$ (Figure~\ref{daisy}). Then $R$ contains the subset $K_S\subset R_S$ on whose complement the involution $\tau$ is defined. Since $g$ maps to the identity in $\D^\infty(R)$, it is ambiently isotopic rel $K_S$ to a diffeomorphism $g'$ of $R$ that is the identity near infinity. Then $g'$ extends by the identity to a diffeomorphism of $X$ restricting to $g$ on $K_S$, showing that $X_\tau$ is also obtained from $X$ by the corresponding twist on $g(K_S)$. But the corresponding copy $g(R_S)$ lies in an $S^4$-summand that is unchanged by the twist, so we obtain the contradiction that $X_\tau\approx X$. The argument is similar for Cases (b) and (c): For the latter, the map $g'$ is the identity on $R$ outside some compact subset $K$, and this subset avoids all but finitely many copies of $R_L$. Throwing away the summands disjoint from $K$, we obtain an exotic $\R^4$ that embeds into a finite blowup of $X$, and the same reasoning applies. For (b), each summand embeds rel $K_S$ in an infinite blowup of $R_S$, embedding $R$ in a blown up infinite end sum of copies of $R_S$. Extend $g'$ over the latter by the identity, blow down all exceptional spheres outside the support of $g'$, and proceed as before.
\end{proof}

\begin{figure}
\labellist
\small\hair 2pt
\pinlabel $X$ at 28 111
\pinlabel $R_S$ at 69 153
\pinlabel $K_S$ at 128 112
\pinlabel $g$ at 133 21
\pinlabel $\R^4$ at 219 112
\pinlabel $\R^4$ at 295 157
\endlabellist
\centering
\includegraphics{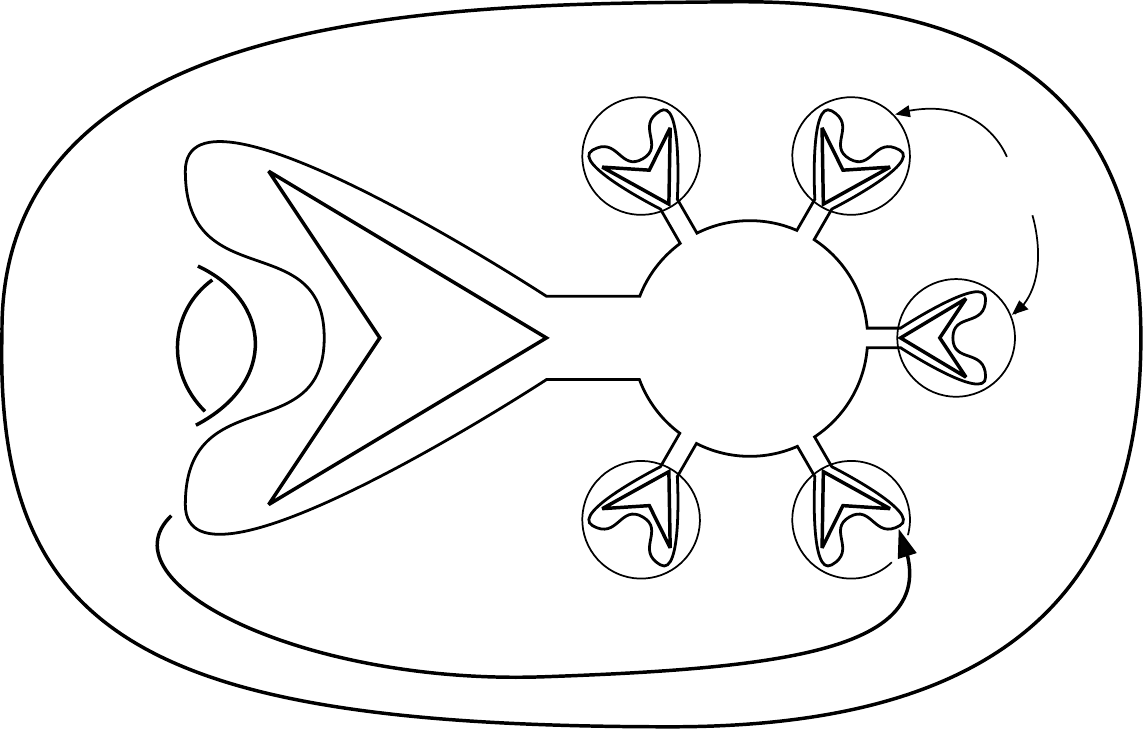}
\caption{Embedding $R$ in $X$.}
\label{daisy}
\end{figure}
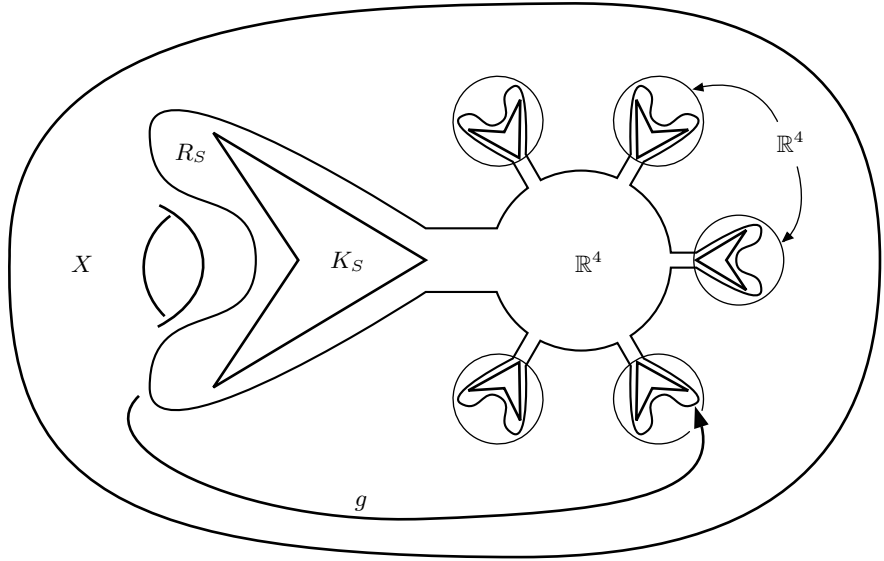

\begin{proof}[Proof of Theorem~\ref{homeo}] 
We use a variation of a trick in \cite{AKMR} for constructing $G$-corks (cf.\ proof of our Corollary~\ref{akmr}), which also inspired the previous proof. By Proposition~\ref{Schoen}, the involution $\tau$ on the small model summand of $R$ extends to a self-homeomorphism of it. For a fixed $s\in \Sigma$, let $h$ be this homeomorphism on the corresponding  small summand of $R$, extended diffeomorphically over $R$ as the identity on the rest of $R$ (after a smooth isotopy rel $K_S$ to fit the pieces together). For $G$ as given, we can modify the $G$-action on $R$ by a diffeomorphism permuting summands, so that $s$ has trivial stabilizer. Suppose there is a nontrivial $g\in G$ whose conjugate $h^{-1}\circ g\circ h$ maps into $r(\D(R))$. Then $h^{-1}\circ g\circ h$  agrees with a diffeomorphism $f\co R\to R$ outside some compact subset $K$. In Case (a), we embed $R$ in $X$ as before and note that twisting by $h$ gives $X_h=X_\tau$. Since $f\circ g^{-1}$ is a self-diffeomorphism of $R$, we have diffeomorphisms $X_h\approx X_{h\circ f\circ g^{-1}}= X_{g\circ h\circ g^{-1}}$. This last twist is given by $\tau$ on the summand indexed by $g(s)$, which lies in an $S^4$-summand of $X$. Thus, it does not change the diffeomorphism type of $X$, and we again have the contradiction $X_\tau\approx X$. For Case (c), let $R'\subset R$ be obtained by deleting all large summands except for the finite collection intersecting the compact set $g(K)$. Then $R'$ has the same preimage under each of the three maps $g$, $h^{-1}\circ g\circ h$ and $f$. (For the first pair this is because $h$ preserves each summand, while the second pair agree outside the compact set $K$.) Embed $R'$ in a finite blowup of $X$ and repeat the previous argument. For (b), the same method works once we replace the model small summand of $R$ by a $\tau$-invariant open subset of it containing $K_S$, with compact closure, and large enough that the substitute for $R'$ contains $g(K)$.
\end{proof}

Now suppose that $G\in\G$ acts on $\R^4$ isometrically with respect to some complete Riemannian metric, and that the corresponding $\Gamma\co \Sigma\times [0,\infty)\to\R^4$ is a multiray (i.e.\ proper). Many countably infinite groups satisfy these conditions, for example, countably generated free groups and properly discontinuous isometry groups of Euclidean or hyperbolic 3-space (Examples~\ref{groups}(b) and (c), respectively).

\begin{thm}\label{isom} For $G\in\G_+$ acting isometrically as above, every exotic $\R^4$ constructed in Theorem~\ref{diffeotopy} admits a complete metric on which $G$ acts isometrically. The action injects into $\D(R)$ and $\D^\infty(R)$ but is topologically equivalent to the original $G$-action on $\R^4$. The same holds for nonorientable actions in Case (a) of the theorem.
\end{thm}

\begin{proof} The multiray is a homeomorphism onto its image, so each ray in $\Gamma$ has a neighborhood in $\R^4$ disjoint from the other rays. For each $s\in \Sigma$, let $U_s$ be a neighborhood of $\gamma_s([0,\infty))$ whose points are within distance 1 of $\gamma_s([0,\infty))$ and closer to it than to any other ray. These neighborhoods are disjoint, cannot cluster (since $\Gamma$ is proper) and can be chosen equivariantly (since $G$ acts by isometries). We can now perform the infinite end sum of Theorem~\ref{diffeotopy} equivariantly, using a tubular neighborhood of each ray $\gamma_s$ with closure in $U_s$. Equivariantly modify the metric in each $U_s$ to merge it with a fixed complete metric on the corresponding exotic $\R^4$ summand. Completeness is preserved if we suitably scale the metric near the end of $R$ along each neck where the metrics were merged.
\end{proof}

This theorem contrasts with Taylor's result that isometry groups of full $\R^4$-homeomorphs must be finite (Theorem~\ref{Taylor}). Note that Theorem~\ref{full}(c) already implies these $\R^4$-homeomorphs cannot be full. Our same techniques also give a contrasting result about full $\R^4$-homeomorphs: 

\begin{thm}\label{fulldiff} There is a doubly uncountable family (as in Lemma~\ref{doubly}(a)) of full $\R^4$-homeomorphs $R$ with uncountable diffeotopy groups $\D(R)$ and $\D^\infty(R)$.
\end{thm}

\begin{proof}
Choose an uncountable group $G\in\G^*$. Theorem~\ref{diffeotopy} gives a small exotic $\R^4$ with a $G$-action fixing a neighborhood of some ray. Use this ray to sum with a single copy of $R_L$. Lemma~\ref{doubly}(a) extends the result to a doubly uncountable family, and $G$ injects into each diffeotopy group as in the proof of Theorem~\ref{diffeotopy}.
\end{proof}

\begin{Remark}\label{unitary} As another variation, consider the finite set  $\G_n$ of groups in $\G_+$ for which $\Sigma$ can be chosen to be finite with $n$ elements. If $\G_+$ is replaced by $\G_n$ in Theorem~\ref{diffeotopy}(c), the same conclusions hold with full $\R^4$-homeomorphs occurring in a doubly uncountable family. The Taylor invariant will be finite but can be chosen arbitrarily large by including sufficiently many summands containing $K_L$. We also obtain uncountably many diffeomorphism types $R$ as in (b) such that for each  subgroup $G$ of ${\rm U}(2)$ with order dividing $n$ there is a $G$-equivariant embedding $R\subset\C^2$ as in (b), where $G$ acts on $\C^2$ unitarily. (Identify the central $\R^4$ of $R$ with the open unit ball in $\C^2$ and sum with an equivariant biholomorphic embedding of $n$ copies of the relevant summand using Eliashberg 1-handles.) We can also adapt (a) by using the analogous finite subset of $\G$. In each case, Theorem~\ref{homeo} applies.
\end{Remark}

Like early techniques for distinguishing $\R^4$-homeomorphs, our methods above are a gateway to analyzing open 4-manifolds in much greater generality. While we do not pursue this systematically here, we state a sample theorem.

\begin{thm}\label{open} For $R_S\subset X$ as in Section~\ref{Review}, suppose a compact topological 4-manifold $Y$ with connected boundary homeomorphically embeds in some blowup of $X$, disjointly from $R_S$. Then there is a smoothing $V$ of $\inter Y$ for which the image of $\D(V)$ in $\D^\infty(V)$ is uncountable and trivially intersects some nonfinitely generated subgroup of $\D^\infty(V)$.
\end{thm}

\begin{proof} Precomposing the embedding by a shrink away from $\partial Y$, we can arrange $Y$ to be disjoint from the closure of $R_S$ and have bicollared boundary. After a further isotopy, we can assume $\partial Y$ is smooth near some point \cite{Q}. We can then ambiently end sum $\inter Y$ with $R$ from our previous proofs, extending the action of a fixed $G\in\G^*$ by the identity on $\inter Y$. The previous proofs then apply.
\end{proof}

As an example, the theorem applies to any $Y$ that topologically embeds in some $\# n\blow$, and the resulting smoothings $V$ realize uncountably many diffeomorphism types by Lemma~\ref{radial}. More general examples can be obtained using more explicit descriptions of $X$. By Theorem~\ref{mainThm} and its addendum, it suffices for $Y$ to avoid an embedding of Akbulut's cork for which twisting changes the Seiberg-Witten invariants (or to avoid any stably effective $G'$-slice cork). Various examples of such embeddings are discussed in Section~\ref{Apps}. The above proof shows that many $G$-actions can be used to exhibit the conclusions of Theorem~\ref{open}. These all arise from actions on $\R^4$. However, more interesting actions on $V$, such as homologically nontrivial actions, can be incorporated by generalizing $\Gamma$-compatibility to smoothings of $\inter Y$.


\section{An example with $r\co\D(R)\to\D^\infty(R)$ surjective}\label{Univ}

Now that we have many $\R^4$-homeomorphs with cokernel $\D^\infty(R)/r(\D(R))$ of maximal cardinality, we ask how small the cokernel can be. Recall that this cokernel is not known to be finite even for the standard $\R^4$ (Theorem~\ref{DR4}). For a better example, we consider $R_U$, the {\em universal $\R^4$} of Freedman and Taylor \cite{FT}, which is characterized by the property that $R\natural R_U\approx R_U$ for every $R$ homeomorphic to $\R^4$.

\begin{thm}\label{univ} Every diffeomorphism of the end of $R_U$ extends over $R_U$. Thus, the restriction $r\co\D(R_U)\to\D^\infty(R_U)$ is surjective.
\end{thm}

\begin{proof}
Given a diffeomorphism of the end of $R_U$, Proposition~\ref{Schoen} extends it to a homeomorphism $f\co R_U\to R_U$. We define a 5-manifold $W$ as a certain smoothing of $I\times \R^4$: First identify the latter with $I\times R_U$. Then change its smoothing near $\{1\}\times\R^4$ by pushing the previous smoothing forward through the homeomorphism $\id_I\times f$. This new smoothing agrees with the old one near infinity. There is a compact region $K\subset\inter W$ on which the smoothings disagree, but we can assume $H^4(\inter W,\inter W-K;\Z_2)=0$, so there is no obstruction to extending the smoothing over all of $W$ \cite{KS}. Now $W$ is a smooth h-cobordism from $R_U$ to itself, with a smooth product structure given near infinity and $\partial W$. The projection to $I$ extends over $W$ as a Morse function with finitely many critical points. As in Casson \cite{C} (also see, e.g., \cite{FQ}, \cite{GS}), we can simplify the corresponding handle decomposition so that it only has 2- and 3-handles, and their belt and attaching spheres in the middle level are algebraically paired, with Casson handles where we need Whitney disks. According to Freedman and Taylor \cite{FT}, we can smoothly cancel the handles provided that we can find certain smooth ``link slice solutions" in the middle level, avoiding the compact region $K'$ containing the belt and attaching spheres and Casson handles. They construct $R_U$ so that it contains the required solutions. We can describe $R_U$ as an infinite end sum of copies of $R_U$, so some summand is disjoint from $K\cup K'$ and contains the required slice solutions. The resulting handle cancellation occurs in a compact region in $\inter W$, so we obtain an identification of $W$ with $I\times R_U$ that agrees with our original product structure near infinity. By construction, $f$ smoothly identifies $R_U$ with $\partial_+ W$. Projecting to $\partial_- W$ gives our diffeomorphism agreeing with $f$ near infinity.
\end{proof}

\begin{prob} Describe the diffeotopy group $\D(R_U)$ of the universal $\R^4$.
\end{prob}

\noindent Note that there is an orientation-reversing involution since $R_U\approx R_U\natural\overline R_U$. If  $\D_+(R_U)$ can be shown to be trivial, we will have exhibited an $\R^4$-homeomorph with $\D_+^\infty(R)$ trivial without having to prove the smooth Schoenflies conjecture.


\section{Corks and $\R^4$-homeomorphs}\label{Corks}

Under broad hypotheses, corks and $\R^4$-homeomorphs can be used interchangeably (up to blowups) for modifying smooth structures on 4-manifolds, and the $\R^4$-homeomorphs can be chosen from doubly uncountable families. We now present this as a theorem, then give applications in Section~\ref{Apps}. Each exotic $\R^4$ arising from Casson's h-cobordism construction is of a special form, called a {\em ribbon $\R^4$} by Freedman \cite{DF}. We slightly generalize the notion and transfer it to corks. The main idea is that a ribbon (or slice) complement in $B^4$ can be turned back into $B^4$ by adding 2-handles to meridians. If we instead attach exotic open 2-handles, the result will still have interior homeomorphic to $\R^4$, but not necessarily diffeomorphic to it. We will usually take the  exotic open 2-handles to be Casson handles as in Casson's original paper \cite{C}, but the generalized Casson handles discussed in \cite{DF} could just as easily be used.

To construct our $\R^4$-homeomorphs and corks, we begin with the union $D=D_1\sqcup\cdots\sqcup D_k\subset B^4$ of an ordered collection of $k$ disjoint, smoothly embedded disks in the 4-ball, so that $D\cap\partial B^4=\partial D=L=L_1\sqcup\cdots\sqcup L_k$ is exhibited as a slice link in $S^3$. Let $E(D)$ denote the (compact) exterior of $D$ in $B^4$, the complement of a tubular neighborhood of $D$. Let $E'(D)\subset E(D)$ be a diffeomorphic copy of $E(D)$ obtained by removing a boundary collar of the latter. For any ordered $k$-tuple $m=(m_1,\cdots,m_k)$ of integers, let $C(D,m)$ be the compact 4-manifold obtained from $E(D)$ by attaching a 2-handle along an $m_i$-framed meridian of each $L_i$. Let $R(D)$ denote an open 4-manifold obtained from $E(D)$ by attaching a Casson handle to a 0-framed meridian of each $L_i$ and deleting the remaining boundary. (In general, this depends on the choices of Casson handles, but these are suppressed from the notation.) When $m=0\in\Z^k$, the manifold $C(D,0)$ is diffeomorphic to $B^4$. Changing the framings does not change the homotopy type, so every $C(D,m)$ is contractible, with boundary the homology sphere obtained from $S^3$ by $-1/m_i$-surgery on each $L_i$. Since every Casson handle is homeomorphic to a 2-handle \cite{F}, every $R(D)$ is homeomorphic (but not necessarily diffeomorphic) to $\R^4$. When $D$ is obtained from an embedded collection of disks in $S^3$ by pushing their interiors into $\inter B^4$, each $C(D,m)$ or $R(D)$ is diffeomorphic to $B^4$ or $\R^4$, respectively, the latter because every Casson handle interior is diffeomorphic to $\R^4$ \cite{F}. (For $R(D)$, it suffices to assume that $L$ is a trivial link \cite[p.\ 464]{BG}.) However, our small exotic $\R^4$-homeomorph $R_S$ from Section~\ref{Review} can be realized as $R(\Delta)$ for the ribbon disk $\Delta$ shown in Figure~\ref{pretzel}, with $K_S=E'(\Delta)$, and uncountably many diffeomorphism types can be realized for $R_S=R(\Delta)$ by varying the choice of Casson handle \cite{BG}. The canonical embedding $R_S\subset S^4$ can be seen directly. In fact, every $R(D)$ canonically embeds in $\R^4$ since every Casson handle canonically embeds in a standard 2-handle \cite{C}. Correspondingly, every $C(D,m)$ with $D$ ribbon canonically embeds in $S^4$. (The double of $C(D,m)$ is $S^4$ since it is obtained from $S^4$ by Gluck twisting some components of the 2-link made by doubling $D$. Such a manifold must be diffeomorphic to $S^4$ by \cite[Exercise~6.2.11(b) and solution]{GS}.)

\begin{figure}
\labellist
\small\hair 2pt
\pinlabel {ribbon move} at -25 73
\endlabellist
\centering
\includegraphics{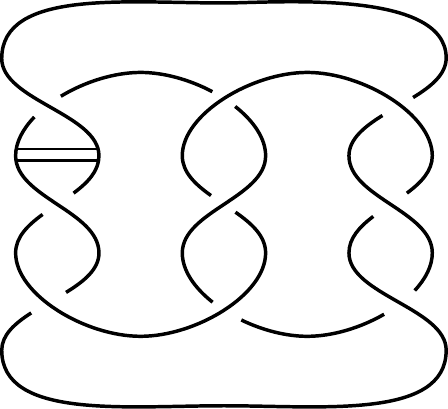}
\caption{The $(-3,3,-3)$-pretzel knot with ribbon move exhibiting the disk $\Delta$.}
\label{pretzel}
\end{figure}

For $D$ as above, let $G$ be a group acting on the pair $(S^3,L)$, preserving the orientation of $S^3$. In general, the action permutes the components of $L$, and hence their indices. Call the action {\em $L$-finite} if the stabilizer of each $L_i$ acts on the circle $L_i$ by factoring through the trivial action or a reflectional $\Z_2$, and factors through a finite action on an $S^1\times D^2$-neighborhood of $L_i$. (In the applications, this action will be trivial or $\Z_2$ reflecting each factor of $S^1\times D^2$.) Such an action canonically extends over $\partial E(D)$ (0-surgery on $L$), although not necessarily over $E(D)$. From $\partial E(D)$, it extends over the 2-handles forming $C(D,m)$ whenever the $k$-tuple $m$ is invariant under the $G$-action. Similarly, we can construct $R(D)$ with an equivariant collection of Casson handles over which the action extends. (Let the attaching regions be the surgery solid tori, and arrange the first stage double points of each fixed sign to be permuted by the stabilizer of each $L_i$.) It is straightforward to further refine the Casson handles equivariantly. In either setting, we obtain an induced $G$-action on the complement of $\inter E'(D)$ that need not extend over $E'(D)$. This clearly restricts to an action on $\partial C(D,m)$ or an action at infinity on $R(D)$.

\begin{de}\label{slice} A pair of the form $(C(D,m),G)$ as above will be called a {\em $G$-slice cork} (or {\em slice $G$-cork}) if for the induced action of $G$ on $\partial C(D,m)$, no element of $G$ other than the identity extends to a diffeomorphism of $C(D,m)$. A pair of the form $(R(D),G)$ as above will be called a {\em $G$-slice $\R^4$} if the corresponding statement holds for the induced $G$-action at infinity, i.e., $G$ injects into the cokernel $\D^\infty(R(D))/r(\D(R(D)))$. We may substitute the word ``ribbon" for ``slice" when $D$ is ribbon (a union of disks whose interiors lack local maxima of the radial function on $B^4$).
\end{de}

\noindent Without the slice condition, $G$-corks were defined in \cite{AKMR} (as contractible manifolds, with boundary $G$-actions satisfying the above nonextendability condition). Casson's h-cobordism construction \cite{C} always yields a $\Z_2$-ribbon $\R^4$. By Proposition~\ref{Schoen}(a,b), the standard $\R^4$ cannot be $G$-slice for a nontrivial group $G$ (and an appropriate $D$) unless the smooth Schoenflies Conjecture is false. Every $G$-action at infinity constructed above on $R(D)$ extends as a {\em topological} $G$-action: By construction, it is homeomorphic to the original $G$-action on $S^3$, extended in the obvious way to a collar of the end of $\R^4$, so it can be extended by coning.

\begin{example}\label{tau} The pretzel knot in Figure~\ref{pretzel} has an obvious $\Z_2\oplus\Z_2$-action whose nontrivial elements $\tau_x$, $\tau_y$ and $\tau_z$ are $\pi$-rotations about the three standard coordinate axes (so $\tau_x$ rotates in the plane of the paper). This is not $L$-finite since $\tau_y$ has no fixed points on the knot, but the two subgroups generated by $\tau_x$ and $\tau_z$ are $L$-finite. (The action is closely related to the $L$-finite $\Z_2\oplus\Z_2$-action on the $(-3,3,-3,3)$-pretzel link that is used to construct the peculiar $\Z_2\oplus\Z_2$-actions of \cite{menag}.) Clearly, $\tau_y$ preserves the ribbon move, so extends over the pair $(B^4,\Delta)$. However, $\tau_x$ and $\tau_z$ cannot extend. In fact, Lemma~\ref{cork} shows that $\tau_x$ induces the usual nonextendable involutions ($\Z_2$-ribbon structures) on the boundary of Akbulut's cork $C(\Delta,-1)$ and the end of $R_S=R(\Delta)$. Since $\tau_z=\tau_x\circ\tau_y$, it also cannot extend. Furthermore, for any fixed embedding of one of these manifolds into $X$, the manifolds $X_{\tau_x}$ and $X_{\tau_z}$ made by regluing will be diffeomorphic to each other, so the two nontrivial involutions can be used interchangeably.
\end{example}

If a manifold $V$ is compact with a $G$-action on its boundary, or open with a $G$-action at infinity, any codimension-zero embedding of $V$ into a manifold $X$ results in a family $\{X_g\st  g\in G\}$ of manifolds obtained by regluing as in Section~\ref{Infty}. Generalizing \cite{AKMR}, we will call an embedding {\em $G$-effective} if for distinct $f,g\in G$, the manifolds $X_f$ and $X_g$ are never diffeomorphic. A pair of the form $(C(D,m),G)$ or $(R(D),G)$ as above that admits a $G$-effective embedding must be a $G$-slice cork or $G$-slice $\R^4$. We need a slightly stronger version of effectiveness:

\begin{de} A $G$-effective embedding $V\subset X$ will be called {\em \blowst $G$-effective} if it remains $G$-effective after $X$ is blown up any (finite) number of times in the complement of $V$.
\end{de}

\noindent Since Donaldson's invariants and their descendants lose no information under blowing up, we lose no generality in practice by assuming stability, provided that we are careful with orientations: When the embedding $V\subset X$ is $G$-effective, so is $\overline V\subset\overline X$. However, $\overline V$ may not admit a stably $G$-effective embedding into any manifold, even if the original embedding $V\subset X$ is stably $G$-effective (Proposition~\ref{flip}).

The main conclusion of this section is that up to blowups, stably $G$-effective embeddings of $G$-slice corks and of $\R^4$-homeomorphs can always be found inside each other, and the resulting set of diffeomorphism types of $\R^4$-homeomorphs has the maximal size discussed in Section~\ref{Review}. Let $D\sqcup n\Delta$ denote the collection of disks in $B^4$ made from a given $D$ by adding $n$ distant copies of the disk $\Delta$ from Figure~\ref{pretzel}, or equivalently, by forming the boundary sum of $(B^4,D)$ with $n$ copies of $(B^4,\Delta)$. Then $E'(D)$ canonically embeds in $E'(D\sqcup n\Delta)$. We write $n\Delta$ when $D$ is empty. Given embeddings $E\subset V_i$ ($i=1,2$) and $G$-actions on $(V_i-\inter E,\partial E)$, we call an embedding $V_1\to V_2$ {\em equivariant along $\partial E$} (resp.\ {\em equivariant except on $E$}) if it restricts to the identity on $E$ and restricts equivariantly to $\partial E$ (resp.\ $V_1-\inter E$). We use the same terminology for embeddings of $V_1$ obtained after blowing up (resp.\ equivariantly blowing up) $V_2$ outside of $E$, although the action need not extend from $\partial E$ to a nonequivariant blowup of $V_2-\inter E$.

\begin{thm}\label{mainThm} Let $(S^3,L)$ be a $k$-component link with an $L$-finite $G$-action, given as slice by disks $D$.

\item{(a)} For each $G$-invariant $m\in\Z^k$, the manifold $C(D,m)\# \blow$ contains some $R(D)$, embedded equivariantly along $\partial E'(D)$, and the blowup is unneccessary if the entries of $m$ are all even.

\item{(b)} For each equivariant $R(D)$, there is an $n\in\Z^+$ such that uncountably many diffeomorphism types of the form $R(D\sqcup n\Delta)$ embed in $R(D)$, equivariantly except on $E'(D)$.

\item{(c)} For each equivariant $R(D)$, there is a finite equivariant blowup $R(D)\# n\blow$ that contains large, full $\R^4$-homeomorphs, comprising a doubly uncountable family as in Lemma~\ref{doubly}(a), embedded equivariantly except on $E'(D)$.

\item{(d)}  For each equivariant $R(D)$ and $G$-invariant $k$-tuple $m$ of sufficiently negative integers, there is an $l\in\Z^+$ such that $C(D,m)$ embeds in $R(D)\# l\blow$, equivariantly except on $E'(D)$.
\end{thm}

\begin{adden}\label{main} In each case of the theorem, if the initial cork or $\R^4$-homeomorph has a stably $G$-effective embedding in some $X$, then each resulting manifold has a stably $G$-effective embedding in the corresponding blowup $X^*$. In particular, all of the corks and small $\R^4$-homeomorphs are then $G$-slice, and $G$ injects into $\D^\infty(R)/r(\D(R))$ for each $\R^4$-homeomorph $R$.
\end{adden}

\begin{proof} All of the discussed embeddings in $X^*$ agree equivariantly along $\partial E'(D)$. For each $g\in G$, the resulting manifold $X^*_g$ is made by cutting out $E'(D)$ and regluing by $g$, so it is independent of the embedding under discussion.
\end{proof}

\begin{proof}[Proof of Theorem~\ref{mainThm}] 
 The last three parts follow easily: Since the action on $S^3$ is $L$-finite, there is a point $p$ in $S^3-L$ with finite orbit whose cardinality we denote by $n$, and stabilizer acting trivially on a neighborhood of $p$. For (b), equivariantly end sum $R(D)$ with $n$ copies of $R_S=R(\Delta)\subset\R^4$ using a multiray along the orbit of $p$, and apply Lemma~\ref{doubly}(a) ($n=0$ case) with $R_r$ in one summand. As discussed following that lemma, we can replace the required radial family in $R_S$ by a suitably nested family parametrized by the Cantor set. Such a family of the required form $R(\Delta)$ exists since every Casson handle contains a Cantor set family of suitably nested Casson handles \cite{F}. (Alternatively, if we use generalized Casson handles as in \cite{DF}, we obtain such a Cantor set family within a radial family.) For (c), also sum with $n$ copies of $R_L\subset\blow$. Fullness follows from Theorem~\ref{full}(b). For (d), note that every Casson handle contains a generically immersed disk spanning its attaching circle. We can make this embedded by blowing up its double points. Blowing up a negative double point does not change the relative homology class of the disk or the framing by which the disk is attached (measured homologically as in Casson handle theory), but blowing up a positive double point lowers the framing by 4. Additional blowups lower the framing by any desired amount. Thus, the required cork $C(D,m)$ can be constructed from $E'(D)$ by equivariantly attaching 2-handles whose cores are such disks. 
 
For (a), note that $C(D,m)-\inter E'(D)$ is simply connected, since $\pi_1(\partial E'(D))$ is generated by meridians to $L$, and these are annihilated when the 2-handles of $C(D,m)$ are attached. We wish to change the cores of these handles, fixing their boundaries $\mu_i$, to get an immersed collection of disks inducing the 0-framing on each $\mu_i$ (in the homological sense) and having vanishing intersection numbers with each other. If the framings $m_i$ induced by the original disks are not all even, we obtain evenness by blowing up once and tubing the odd disks into the resulting odd-framed sphere. The meridians $\mu_i$ form a basis for $H_1(\partial E'(D))$, so we can find a dual basis $\{\beta_j\}$ for $H_2(\partial E'(D))$ with $\mu_i\cdot\beta_j=\delta_{ij}$. Since  the complement of $\inter E'(D)$ is simply connected, we can represent the image of this dual basis inside it by immersed spheres. Since the intersection pairing vanishes on the span of these spheres, we can achieve the required form for our disks by suitably tubing them into copies of the spheres. Casson's embedding theorem \cite{C} now produces the required Casson handles. (Alternatively, we can appeal to the Disk Theorem \cite{FQ} and the fact \cite{Q} that every topological 2-handle contains a Casson handle.) After passing to a suitable common refinement of the Casson handles, we can extend the $G$-action over them to obtain the required $(R(D),G)$. 
\end{proof}

\begin{Remarks}\label{mainrem} (a) To construct a  \blowst $G$-effective embedding of a $G$-slice cork or $\R^4$ with a given $D$, it is enough to construct a \blowst $G$-effective embedding $E(D)\subset X$  with $X-\inter E(D)$ simply connected. Then there is a generically immersed collection of disks in $X-\inter E(D)$ with the required meridian boundaries, and the previous methods apply.

\item[(b)] An argument similar to (d) above shows that every compact subset $K$ of $R(D)$ lies in some equivariant $C(D',m)$ inside a blowup of $R(D)$, where $D'$ is an iterated, ramified Whitehead double of $D$. This is because $K$ lies in the union of $E(D)$ with a finite subtower of each Casson handle, so we can blow up the double points at a higher stage of each. It follows that an infinite blowup of $R(D)$ is a nested union of blown up contractible manifolds. In contrast, this statement fails for any large exotic $\R^4$ containing $\overline K_L$, since if the latter were contained in a blowup of a contractible $C$, capping with $\overline C$ would put it in a closed, negative definite manifold. Furthemore, the interior of a blown up contractible manifold with boundary a nontrivial homology sphere cannot be a nested union of blown up $\R^4$-homeomorphs, since the latter must be simply connected at infinity.

\item[(c)] Suppose $E(D)$ admits a Stein structure. Then so does every $C(D,m)$ with sufficiently negative framings $m_i$, so the corks produced by Theorem~\ref{mainThm}(d) can be assumed to be Stein. The $\R^4$-homeomorphs in (a) and (b) can also be assumed to be Stein, after equivariantly refining the Casson handles to have a suitable excess of positive double points, cf.\ \cite{Ann}. If $X$ in the addendum is complex, then the resulting Stein surfaces in each case can be assumed to be biholomorphically embedded in $X^*$ since they are contractible \cite{steindiff}. The first example of such a Stein exotic $\R^4$ \cite{Ann} had the form $R(\Delta)$ as in Example~\ref{tau}. This was proved by showing that $E(\Delta)$ is Stein, and the meridian was controlled so that the Casson handle could be taken to have just one double point (positive) at each stage. That same control implies that $C(\Delta,m)$ is Stein whenever $m<0$. (For $m=-1$ we recover Akbulut's cork, which is well-known to be Stein by a more direct method.)
\end{Remarks}

It remains to verify orientation-sensitivity of stable $G$-effectiveness, as well as to supply the postponed proof of Lemma~\ref{doubly}(b). The hypotheses of the following proposition are satisfied by the various explicit examples of $G$-corks in the literature with their natural orientations  (see Section~\ref{Apps}), as well as by the corresponding small $\R^4$-homeomorphs generated from them by Addendum~\ref{main}, such as $R(\Delta)$.

\begin{prop}\label{flip} Let $(S^3,L)$ be a link with a nontrivial $L$-finite $G$-action, exhibited as slice by disks $D$.

\item[a)] Whenever $m$ has no positive entry $m_i$, $\overline{C(D,m)}$ admits no stably $G$-effective embedding.

\item[b)] If some $R(D)$ admits a stably $G$-effective embedding, then there is another choice of Casson handles for which the resulting $R(D)$ admits such an embedding but its mirror image $\overline{R(D)}$ does not.
\end{prop}

\noindent In contrast, a stably $G$-effective, $G$-slice cork or $\R^4$-homeomorph can typically be modified so that it has an orientation-reversing involution, by summing it with its mirror image at a fixed point of the action on $S^3-L$. (Embed the new summand in a ball in the associated closed manifold. Then its twisting does not affect the Seiberg-Witten invariants.) A stably $\Z$-effectively embedded amphichiral $R(D)$ with $D$ a single ribbon disk is given in Corollary~\ref{Zcork}; see Remark~\ref{Zrem}(c).

\begin{proof}
Given an embedding  $\overline{C(D,m)}={C(\overline D,-m)}\subset X$, we can blow up to reduce the framings of the 2-handles, obtaining an embedding of ${C(\overline D,0)}$ with the same restriction to $E(\overline D)$. However, this is just $B^4$ with the standard embedding of $E(\overline D)$. Every $g\in G$ extends over $B^4-E'(\overline D)$. Thus, twisting by $g$ is the same as removing and replacing $B^4$, which preserves the ambient diffeomorphism type.

Given a stably $G$-effective embedding of some $R(D)$, we obtain another such embedding by equivariantly blowing up to eliminate all negative first-stage double points of the Casson handles without disturbing their framings. The mirror image of this new $R(D)$ has only negative double points at the first stage, so the previous argument applies to any embedding of $\overline{R(D)}$.
\end{proof}

\begin{proof}[Proof of Lemma~\ref{doubly}(b)]
We construct uncountably many diffeomorphism types of Stein $\R^4$-homeomorphs of the form $R=\natural_\infty R^+_r$ admitting no orientation-reversing self-diffeomorphisms. Let $\{R_r\}$ be a nested family of the form $R(\Delta)$ in $R_S$, indexed by the Cantor set as in the proof of Theorem~\ref{mainThm}(b). For each $r$, Let $R^+_r$ be the Stein $\R^4$-homeomorph obtained from $R_r$ by simplifying its Casson handle by eliminating all of the negative double points at each stage. For $r<s$ we immediately have embeddings $R_r\subset R_s\subset R^+_s$ restricting to the identity on $K_S=E'(\Delta)$, where the first image has compact closure. It follows that $R^+_s$ cannot be diffeomorphic to $R^+_r$ fixing $K_S$. Otherwise, we would have $R_r\subset R^+_r\subset R_r\#\infty\blow$, with the last embedding obtained by blowing up all negative double points of $R_r$. (The notation refers to the unique manifold obtained by blowing up a closed, discrete, infinite set of points.) Since $\cl R_r$ is compact, we would then have $R_r$ embedded with compact closure rel $K_S$ in a finite blowup of itself. This is the embedding needed for deriving a contradiction from periodic end theory as in \cite[Lemma~1.2]{menag}. We immediately obtain uncountably many diffeomorphism types of manifolds $R^+_r$. The same argument applies after sums with other slice $\R^4$-homeomorphs, since these embed in $\R^4$. In particular, we obtain uncountably many diffeomorphism types of Stein $\R^4$-homeomorphs $R$, each an infinite end sum of copies of a fixed $R^+_r$. 

We show that such an $R$ cannot have an orientation-reversing self-diffeomorphism. In fact, $\overline R\#\infty\blow$ is a nested union of blown up balls (cf.\  previous proof and Remark~\ref{mainrem}(b)). However, $R\#\infty\blow$ cannot be such a union. Otherwise, the embedding $K_S\subset R$ determined by a given end summand $R^+_r$ would, after blowups, factor through an embedding of a blown up ball. By compactness, only finitely many blowups would be needed. But $R$ embeds in $R^+_r$ rel $K_S$ since all the other summands of $R$ embed in $\R^4$, and $R^+_r$ embeds rel $K_S$ in $R_r\#\infty\blow$. Thus, the embedding $K_S\subset R_r$ would also factor through a blown up ball after finitely many blowups. Composing with the defining embeddings $R_r\subset R_S\subset X$, we would conclude that $K_S$ lies in a negative definite connected summand of some blowup $X\# n\blow$. This contradicts the fact that $X$ and $X_\tau$ have different Seiberg-Witten invariants (cf.\ Proposition~\ref{exotic}). The lemma follows immediately. (We can easily adjust the embedding $R_r\subset R_S$ by an isotopy so that the negative double points of $R_r$ do not cluster in $R_S$.)
\end{proof}

\begin{Remark}\label{nonunique} As usual for small $\R^4$-homeomorphs, and unlike the large case, we cannot conclude that the manifolds $R^+_r$ are pairwise nondiffeomorphic, but only that they are nondiffeomorphic rel $K_S$, so the diffeomorphism types $R^+_r$ realize the cardinality of the continuum in ZFC set theory.  This phenomenon seems unavoidable. For example, given a compact subset $K\subset R$ of a small exotic $\R^4$, one can always construct a radial family containing $K$ that has several diffeomorphic members. (Given one radial family with $K\subset R'\subset R''\subset R$, we can end sum with a suitable radial family in $\R^4$ to get $K\subset R'\natural R''\subset R''\natural R'\subset R$ with $K$ in the first summand. The obvious diffeomorphism sends $K$ to the second summand.)
\end{Remark}


\section{$G$-effective, $G$-ribbon $\R^4$-homeomorphs}\label{Apps}

There are various examples in the literature of $G$-effective embeddings of $G$-corks. These typically can be seen to be $G$-ribbon corks (and stably $G$-effective), so Theorem~\ref{mainThm} and Addendum~\ref{main} immediately transform these to stably $G$-effective embeddings of $G$-ribbon $\R^4$-homeomorphs, arising with uncountably many (small) diffeomorphism types and extending to doubly uncountable families of (stably $G$-effectively embedded) large $\R^4$-homeomorphs. This section presents such corollaries of Theorem~\ref{mainThm}, without continued mention of uncountability. As discussed in Example~\ref{tau}, the small $\R^4$-homeomorph $R_S$ used throughout the paper is related in this way to Akbulut's original $\Z_2$-cork, which appears in the literature with many $\Z_2$-effective embeddings. Other $G$-corks with finite $G$ are given as
 $l$-fold boundary sums of copies of Akbulut's cork, so Lemma~\ref{doubly}(a) ($n=0$ case) implies the corresponding small $\R^4$-homeomorphs arise in uncountable families of the form $R(l\Delta)$ (without increasing $l$ through Theorem~\ref{mainThm}(b)), and these can all be assumed to be Stein. We first discuss these examples, then turn to infinite order corks, which arise from an entirely different construction. We begin by verifying that Akbulut's cork is the ribbon $\Z_2$-cork $C(\Delta,-1)$ \cite{BG}, and identify the involution. See \cite[Section~9.3]{GS} for more details on this cork, $R(\Delta)$ and their relation to h-cobordisms.

\begin{lem}\label{cork} Akbulut's cork $W$ is equivariantly diffeomorphic to the ribbon $\Z_2$-cork $C(\Delta,-1)$, where $\Delta$ is the ribbon disk shown in Figure~\ref{pretzel}, and its boundary involution $\tau_x$ is $\pi$-rotation in the plane of the paper. The same involution exhibits some $R(\Delta)$ as $R_S$ (as defined in Section~\ref{Review}).
\end{lem}

\begin{proof} Figure~\ref{hCob}(a) shows a 1-handle and 2-handle that cancel to give a manifold $B$ diffeomorphic to $B^4$. There is a slice disk $D\subset B$, obtained from the obvious disk in $S^3$ avoiding the dotted circle, by pushing its interior into $\inter B^4$. The pair $(\partial B,\partial D)$ is preserved by $\pi$-rotation $\tau$ about the $z$-axis. This figure first appeared in \cite{BG}, arising from the h-cobordism generating $R_S$, and exhibiting $R_S$ as $R(D)$ with involution $\tau$. That paper also observed that the pair $(B,D)$ is diffeomorphic to $(B^4,\Delta)$ (as we are about to check). We obtain $E(D)$ by putting a dot on $\partial D$.  Adding a $(-1)$-framed meridian of $\partial D$ and equivariantly canceling, we obtain $C(D,-1)$ in Figure~\ref{hCob}(b), which is a known picture of Akbulut's cork with its involution (e.g.\ \cite{GS}). To identify $(B,D,\tau)$ as $(B^4,\Delta,\tau_x)$, equivariantly isotope  Figure~\ref{hCob}(a) to simplify the 1-2 pair, obtaining Figure~\ref{ED}(a). Interpreting the diagram as a 3-manifold and equivariantly sliding, then canceling the pair, gives the $\Z_2$-invariant knot of Figure~\ref{ED}(b). If we interpret the same computation nonequivariantly on the 4-manifold level, the cancellation follows a pair of dotted circle slides (given by the right arrow), introducing a pair of  ribbon moves. These represent saddle points in $\inter B^4$, so the 2-handle can be passed through them to cancel the 1-handle. The ribbons ultimately can be seen to be parallel, so one cancels against the local minimum between them. The remaining ribbon move appears in Figure~\ref{ED}(b). To complete the computation, fold the leftmost pair of arches of that diagram down by a $\pi$-rotation about the $y$-axis, leaving the rest of the knot fixed. This preserves the involution if we simultaneously rotate its axis from the $z$-axis to the $x$-axis. The result is equivariantly planar isotopic to Figure~\ref{pretzel} with involution $\tau_x$.
\end{proof}

\begin{figure}
\labellist
\small\hair 2pt
\pinlabel {(a)} at 0 0
\pinlabel {(b)} at 200 0
\pinlabel $0$ at 154 80
\pinlabel $0$ at 351 80
\pinlabel $\partial D$ at 39 80
\pinlabel $\tau$ at 63 178
\pinlabel $\tau$ at 260 178
\endlabellist
\centering
\includegraphics{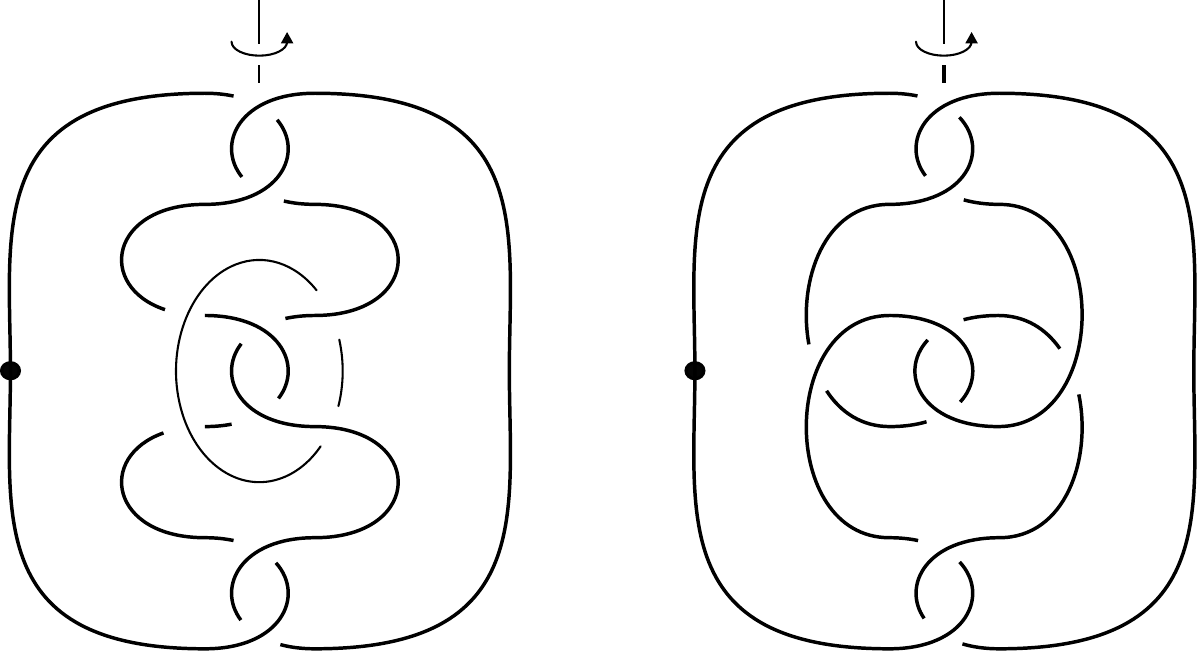}
\caption{(a) The ribbon disk $D$ and (b) Akbulut's cork $W$.}
\label{hCob}
\end{figure}
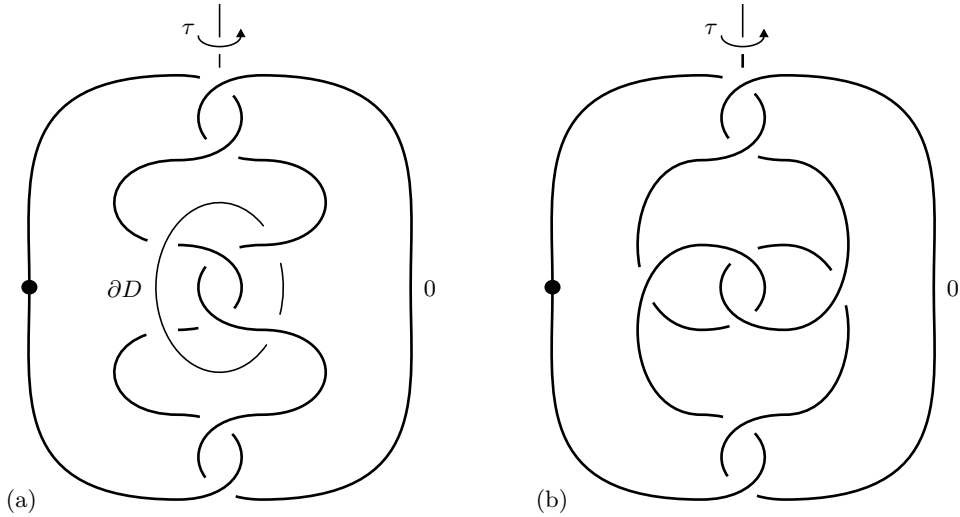

\begin{figure}
\labellist
\small\hair 2pt
\pinlabel {(a)} at 0 0
\pinlabel {(b)} at 188 0
\pinlabel $0$ at 93 114
\pinlabel $\tau$ at 61 123
\pinlabel $\tau$ at 257 123
\pinlabel $\partial D$ at 127 21
\endlabellist
\centering
\includegraphics{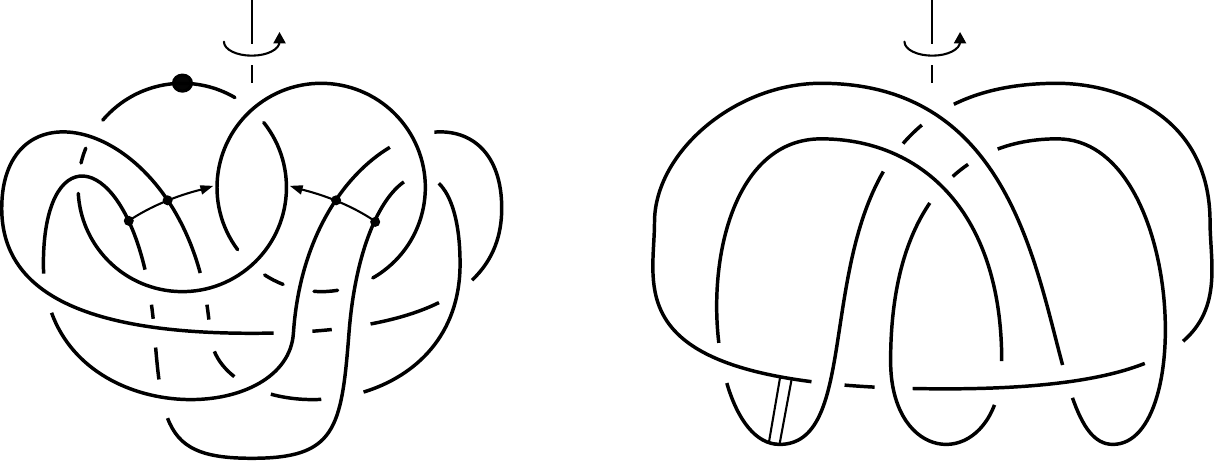}
\caption{Transforming $D$ to $\Delta$.}
\label{ED}
\end{figure}

Akbulut's cork is the simplest of an infinite family $\{W_n\st n\in\Z^+\}$ of $\Z_2$-corks arising in several papers of Akbulut and Yasui, e.g.\ \cite{AY}. We digress to observe that each of these yields an uncountable family of exotic $\Z_2$-ribbon $\R^4$-homeomorphs:

\begin{prop} Each $W_n$ is a $\Z_2$-ribbon cork of the form $C(D,m)$ with $m=(-1,\dots,-1)\in\Z^n$. 
\end{prop}

\begin{proof}
The lower right diagram of Figure~\ref{Wn} shows $W_n$, with its boundary involution given by $\pi$-rotation about the $z$-axis. The twist boxes count right {\em half}-twists. (To see that $W_1=W$, compare the bottom four half-twists with the top and bottom clasp of Figure~\ref{hCob}(b).) To show that $W_n$ is equivariantly diffeomorphic to the left diagram, cancel the $n$ concentric handle pairs of the latter from the inside out: First absorb the lowest half-twist of the lowest $\beta$ into the twist box below it. Then equivariantly cancel the handle pair, which eliminates the two lateral half-twists from $\beta$. Finally, absorb the remaining half-twist of $\beta$ into the twist box below it. To see that the left diagram is obtained from a ribbon complement in $B^4$ by adding ($-1$)-framed meridians, note that the dotted circles comprise an unlink, and that erasing the concentric handle pairs leaves a Hopf link. (The topmost $\beta$ and the lowest half-twist from each $n$-twist box together comprise a global half-twist that can be removed by a flype of the top half of the picture about the $z$-axis.)
\end{proof}

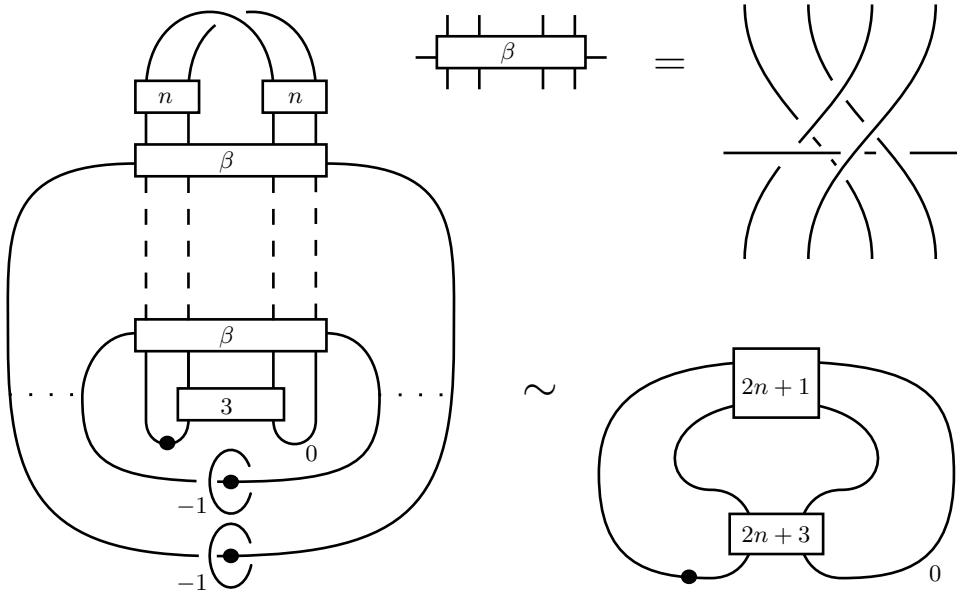
\begin{figure}
\labellist
\small\hair 2pt
\pinlabel $2n+3$ at 290 20
\pinlabel $2n+1$ at 290 76
\pinlabel $\beta$ at 189 202
\pinlabel $\beta$ at 83 161
\pinlabel $\beta$ at 83 95
\pinlabel $n$ at 60 185
\pinlabel $n$ at 109 185
\pinlabel $3$ at 83 69
\pinlabel $-1$ at 70 31
\pinlabel $-1$ at 70 2
\pinlabel $0$ at 115 51
\pinlabel $0$ at 350 6
\pinlabel {\huge $=$} at 250 197
\pinlabel {\huge $\sim$} at 201 74
\endlabellist
\centering
\includegraphics{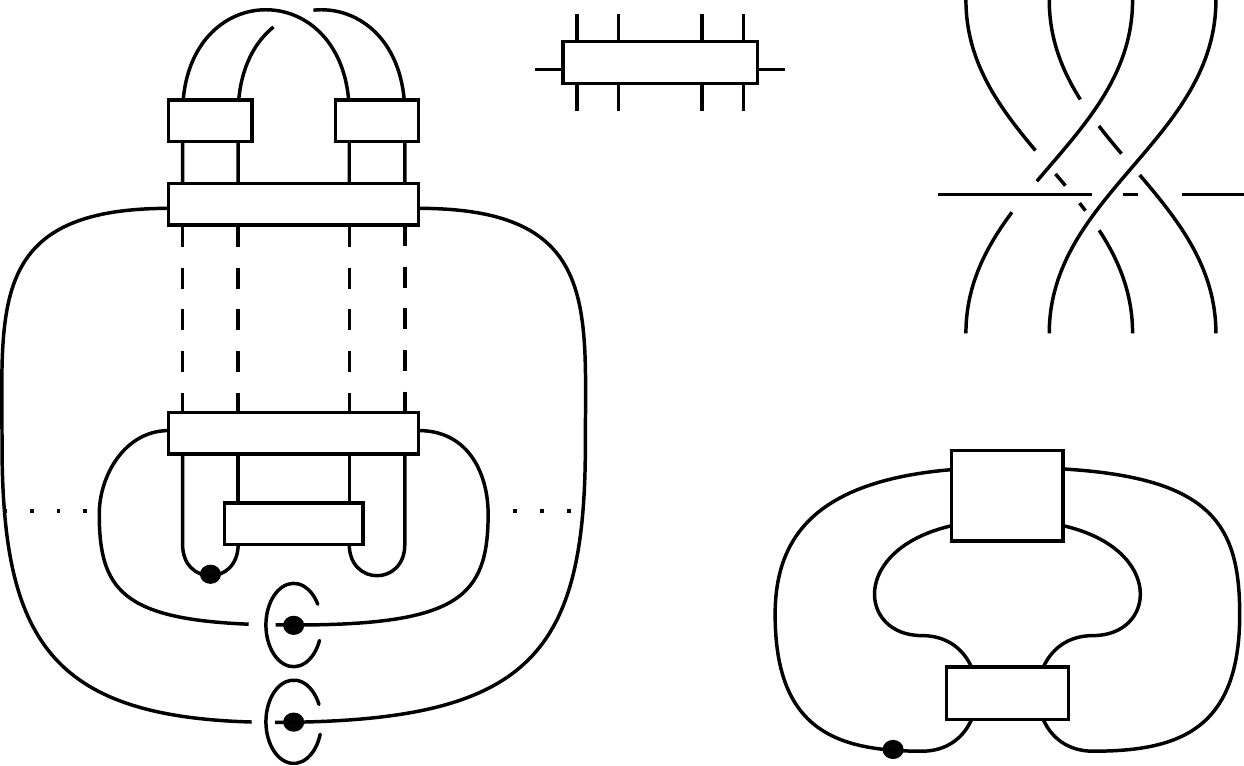}
\caption{The Akbulut-Yasui cork $W_n$ as a $\Z_2$-ribbon cork. The boxes count right {\em half}-twists, except for the $n$ boxes $\beta$ representing the pictured tangle.}
\label{Wn}
\end{figure}

Our first corollary on $G$-effective embeddings is obtained by transforming results of Auckly, Kim, Melvin and Ruberman \cite{AKMR}, notably the existence of $G$-corks of all finite orders.

\begin{cor}\label{akmr} (a) For every subgroup $G$ of $\SO(4)$ with finite order $n$, there is a \blowst $G$-effective embedding of a $G$-ribbon Stein exotic $\R^4$ of the form $R(n^2\Delta)$ into a closed, simply connected 4-manifold.
\item{(b)} For each $n\in\Z^+$, there is a $\Z_2$-ribbon Stein exotic $\R^4$ of the form $R(n\Delta)$, with $n$ disjoint embeddings into a fixed closed, simply connected 4-manifold $X$, such that the $n$ manifolds obtained from $X$ by twisting a single one of the images of $R(n\Delta)$ are all pairwise nondiffeomorphic.
\end{cor}

\begin{proof} The corresponding theorem for corks is proved in \cite{AKMR}, with $R(n\Delta)$ replaced by the $n$-fold boundary sum of copies of $W$, which is $C(n\Delta,m)$ with each $m_i=-1$. They first prove the analog of (b), using an example of Akbulut and Yasui \cite{AY} to construct $n$ disjoint embeddings of $C(n\Delta,m)$ for which the corresponding twisted manifolds are distinguished by the number of their Seiberg-Witten basic classes. Our (b) then follows immediately from Theorem~\ref{mainThm} and the proof of Addendum~\ref{main}. For the cork version of (a), they ambiently connect each of their corks to a single 4-ball, to get an embedded $C(n^2\Delta,-1,\dots,-1)$ with a $G$-action. As in the proof of our Theorem~\ref{homeo} (which was inspired by the argument in \cite{AKMR}), they twist the ambient manifold by $\tau$ on one $C(n\Delta,m)$  summand, obtaining a new manifold with an embedded cork for which the boundary $G$-action is stably effective. We immediately deduce (a).
\end{proof}

\begin{Remarks} (a) In each case above, the induced group action at infinity on the exotic $\R^4$ is homeomorphically conjugate to the restriction of the standard linear action on $\R^4$ (as in Theorem~\ref{homeo} for (a) and the text following Definition~\ref{slice} for (b)).

\item{(b)}  To draw the $G$-corks constructed above, start with the $G$-action on $S^3$. In one fundamental domain, draw $n$ copies of Figure~\ref{hCob}(b), then transport these by $G$ to get $n^2$ copies. This shows the boundary $G$-action. To see the 4-manifold, interchange dots and zeroes in one fundamental domain. For $R(n^2\Delta)$, apply the same procedure to Figure~\ref{pretzel} with a suitably complicated Casson handle equivariantly attached to its meridian, and switch the ribbon move to the opposite side of the knots in one fundamental domain.
\end{Remarks}

Some finite groups $G$ cannot act effectively on any homology 3-sphere, e.g.\ \cite{Z}, in which case there can be no $G$-cork. For example, this holds for  $G=(\Z_2)^n$ with $n>4$ (or $n>3$ in our orientation-preserving setting) \cite[Proposition~3]{Z}. For this group, \cite{AKMR} constructs an effective {\em weak $G$-cork}, where $G$ is a subgroup of the boundary  diffeotopy group that need not lift to a $G$-action. In the $\R^4$ case, we have:

\begin{cor} If a group $G$ cannot act effectively on $S^3$, then there is no $G$-slice $\R^4$. But every $G\cong(\Z_2)^n$ injects into $\D^\infty(R(4^n\Delta))$ (for a suitably chosen Casson handle), such that for some embedding into a closed, simply connected 4-manifold $X$, the manifolds $X_g$ for all  $g\in G$ are pairwise nondiffeomorphic. In particular, the images of $G$ and $\D(R(4^n\Delta))$ intersect trivially.
\end{cor}

\begin{proof} The first sentence follows immediately from the definition of a $G$-slice $\R^4$. The cork version of the rest is proved in \cite{AKMR} by locating $2^n$ disjoint $\Z_2$-corks $C(2^n\Delta,m)$ in a suitable $X$ by their analog of Corollary~
\ref{akmr}(b), then connecting them as leaves on a thickened binary tree to get an embedding of  $C=C(4^n\Delta,m)$. The group $(\Z_2)^n$ acts on the tree, freely on the leaves. Its generators lift to a $\Z^n$-action on $C$ that factors through $(\Z_2)^n$ to $\D(C)$. (Every even element is a twist on disjoint 3-disks in $B^4$ that can be undone by an isotopy avoiding the leaves.) As before, twisting on one leaf gives the required cork. Theorem~\ref{mainThm}(a) with $G=\Z^n$ gives the corresponding ribbon $\R^4$.
\end{proof}

We also have an exotic $\R^4$ version of \cite[Theorem~1.2]{AY}:

\begin{cor} For any fixed $n\ge2$, there is an exotic Stein $\Z_2$-ribbon $\R^4$ of the form $R(\Delta)$ and a family of (nondisjoint) embeddings of it into $X=\#(2n-1)\CP^2\#(10n+1)\blow$ for which the manifolds obtained by twisting realize infinitely many diffeomorphism types. In fact, for any knot $K\subset S^3$, the manifold $E(n)_K\#2\blow$ obtained from the elliptic surface $E(n)$ by the Fintushel-Stern knot construction and blowing up can be obtained in this manner.
\end{cor}

\begin{proof} Akbulut and Yasui start with an embedding of $W$ into $E(n)\#\blow$. They show that performing the Fintushel-Stern construction on $E(n)\#\blow$ using a fiber $F$ disjoint from $W$ and any knot, followed by twisting $W$, always yields the same manifold $X'$. In fact, this $X'$ is $\#(2n-1)\CP^2\#10n\blow$ \cite[Exercise~9.3.4]{GS}, giving the cork analog of the theorem. Theorem~\ref{mainThm} and Addendum~\ref{main} complete the proof after a single blowup of $X'$. (In fact, this last blowup can be avoided by using the proof of Theorem~\ref{mainThm}(a) for even $m$, after tubing the 2-handle core in the $\Z_2$-cork into an immersed odd sphere, since $E(n)\#\blow-F-W$ is simply connected and not spin.)
\end{proof}

Perhaps this theorem and its cork analog are not surprising. For simply connected, closed 4-manifolds, there seems to be a sense in which ``most" cut-and-paste constructions (outside a narrow range of standard operations) yield connected sums of copies of $\pm \CP^2$. This raises the following question.

\begin{ques} Is there a manifold $X$ whose Seiberg-Witten invariants are nontrivial, a fixed $\Z_2$-cork or exotic $\R^4$ with involution at infinity, and a family of embeddings of it in $X$ for which twisting yields infinitely many diffeomorphism types?
\end{ques}

Now we turn to a very different construction of corks. An easy way to construct a ribbon disk is to start with any nontrivial knot $\kappa$, interpreted as a tangle in the 3-ball, and cross the pair with $I$. The result is a ribbon disk $D_\kappa$ for the knot $\kappa\# -\kappa$. For each nonzero $m\in\Z$, the homology sphere $\partial C(D_\kappa,m)$ has an incompressible torus $T$, namely the boundary of $B^3$ minus a tubular neighborhood of $\kappa$. There is a self-diffeomorphism $f$ of $\partial C(D_\kappa,m)$, supported in a product neighborhood $I\times T$ of $T$, that rotates each $\{t\}\times T$ through an angle $2\pi t$ parallel to the longitude of $\kappa$ (up to smoothing near $t=0,1$). For the double twist knot obtained from the lower right diagram of Figure~\ref{Wn} by replacing the two twist boxes with $r$ right and $s$ left full twists, respectively, let $D_{r,s}$ be the corresponding ribbon disk. The main theorems of \cite{InfCork} state that for $r,s>0>m$, $f$ generates a boundary $\Z$-action for which $C(r,s;m)=C(D_{r,s},m)$ is a $\Z$-cork with a $\Z$-effective embedding in a closed, simply connected 4-manifold $X$. Unlike for previous examples, when none of $r,s,|m|$ is 2, we can assume each $X_{f^n}$ is {\em irreducible}, so not a connected sum of two manifolds with nontrivial homology. In particular, it is not a blowup. The corresponding theorem for ribbon $\R^4$-homeomorphs is the following:

\begin{cor}\label{Zcork} For every $r,s\in\Z^+$, there exists a $\Z$-ribbon exotic $\R^4$  of the form $(R(D_{r,s}),f)$ with a \blowst $\Z$-effective embedding into a closed, simply connected 4-manifold $X$. If neither $r$ nor $s$ is 2, then  each $X_{f^n}$ can be assumed irreducible.
\end{cor}

\begin{proof} Apply Theorem~\ref{mainThm}(a) and Addendum~\ref{main} to the given embedding $C(r,s;-4)\subset X$. Since $m=-4$ is even, no blowing up is necessary. (To verify that the cork is {\em stably} $\Z$-effectively embedded as required, note that the twisted manifolds $X_{f^n}$ were also obtained from $X$ by the Fintushel-Stern knot construction, so distinguished by their Seiberg-Witten invariants via a theorem of Sunukjian \cite{S}. The same argument applies after blowing up $X$.)
\end{proof}

\begin{Remarks}\label{Zrem} (a) The $\Z$-action on $(S^3,\partial D_{r,s})$ is trivial away from $T$, so we can take $n=1$ in Theorem~\ref{mainThm}(b) and (c). We obtain uncountably many diffeomorphism types of the form $R(D_{r,s}\sqcup \Delta)$, $\Z$-effectively embedded in $X$ (irreducible for $r,s\ne2$), and a doubly uncountable family in $X\# \blow$. It seems likely that there are already uncountably many of the form $R(D_{r,s})$ in $X$. To prove this, it would suffice to adapt the end-periodic gauge theory of \cite{DF} to a pair of manifolds related by a twist on an embedded $R(D_{r,s})$ (as was done in \cite{menag} for a version of $R_S$ in $K3\#\blow$).

\item[(b)] Unlike the manifolds made from $n\Delta$, it is unclear whether any $R(D_{r,s})$, $E(D_{r,s})$ or $C(D_{r,s},m)$ admits a Stein structure.

\item[(c)] Since every $D_\kappa\subset I\times B^3$ is preserved under reflection of the $I$-factor, we can arrange $R(D_{r,s})$  to admit an orientation-reversing self-diffeomorphism by suitably refining the Casson handle. This contrasts with the chirality typically underlying 4-dimensional constructions. The fact that the $r$- and $s$-twists of the double twist knot must be in opposite directions is also surprising. Note that if we use $+1$-twists in both places (so $\kappa$ is a trefoil), all boundary diffeomorphisms of the resulting contractible $C(1,-1;-1)$ extend over it \cite{Cork2}, so it cannot be a $G$-cork for any nontrivial $G$.

\item[(d)] The end of each $R(D_\kappa)$ displays structure that isn't apparent in more general $\R^4$-homeomorphs. The incompressible torus $T$ in $\partial E(D_\kappa)$ determines a proper embedding $\tilde T$ of $T^2\times[0,\infty)$ into $R(D_\kappa)$, near which the above $\Z$-action at infinity is supported. The pair $(R(D_\kappa),\tilde T)$ is homeomorphic at infinity to $(S^3,T)\times[0,\infty)$, diffeomorphically away from $(\kappa\# -\kappa)\times[0,\infty)$. The end of $R(D_\kappa)$ is split by $\tilde T$ into two pieces. One side is diffeomorphically $(S^3-\nu\kappa)\times[0,\infty)$, so $\tilde T$ is incompressible there in the usual sense of $\pi_1$-injectivity (at infinity). The other side is homeomorphic to $S^1\times D^2\times[0,\infty)$, so $\pi_1$-injectivity fails. However, there can be no smooth spanning solid tori near infinity in $R(D_{r,s})$ ($r,s>0$) by Proposition~\ref{exotic}. In comparison, our previous constructions depended on end sums, so splitting along proper embeddings of $\R^3$, or of $S^2\times[0,\infty)$ at infinity. It is unclear whether any of this structure has deeper significance or is just an artifact of the constructions. One might hope that such splitting 3-manifolds at infinity play a role analogous to essential spheres and incompressible tori in studying diffeotopy of 3-manifolds. However, it seems unlikely that there is any analog of JSJ-decompositions.
\end{Remarks}

Tange \cite{Tn3} recently showed that the above $\Z$-corks can be boundary-summed together to create $\Z^k$-corks with (stably) $\Z^k$-effective embeddings, generated by $k$ incompressible tori as above. He also gave restrictions \cite{Tn2} on what families can be realized by twisting on $G$-corks, most notably for infinite groups $G$. These results translate to ribbon $\R^4$-homeomorphs:

\begin{cor} (a) For every $k\in\Z^+$, there exists a $\Z^k$-ribbon exotic $\R^4$ with a \blowst $\Z^k$-effective embedding into a closed, simply connected 4-manifold.

\item[(b)] Suppose a family of homeomorphic closed, oriented 4-manifolds with $b_+\ge3$ realizes infinitely many isomorphism classes of mod 2 Ozsv\' ath-Szab\' o invariants. Then the family cannot be realized by twisting a fixed embedding of a $G$-slice $\R^4$.
\end{cor}

\noindent As \cite{Tn2} notes, families as in (b) can be easily made, for example, by the Fintushel-Stern construction.

\begin{proof}
For (a), apply the previous method to Tange's examples, obtaining a stably $\Z^k$-effective embedding of $R(\sqcup^k_{i=1} D_{r_i,s_i})$. (Tange exhibits the corks when $r_i=s_i=1$, but any positive integers can be realized by blowing up as in \cite{InfCork}.) If the family in (b) could be realized by twisting a $G$-slice $\R^4$, then by Theorem~\ref{mainThm}(d), it could be realized by a $G$-slice cork after blowups, contradicting \cite{Tn2}. 
\end{proof}



\begin{thebibliography}{MM}

\bibitem[A]{A}
S.\ Akbulut,
{\em An exotic 4-manifold},
J.\ Differential Geom.\ {\bf33} (1991), 357--361.

\bibitem[AY]{AY}
S.\ Akbulut and K.\ Yasui,
{\em Knotting corks},
J.\ Topol.\ {\bf2} (2009), 823--839.

\bibitem[AKMR]{AKMR}
D.\ Auckly, H.\ Kim, P.\ Melvin and D.\ Ruberman,
{\em Equivariant corks},
arXiv:1602.07650 .

\bibitem[B]{B}
J.\ Bennett,
{\em Exotic smoothings via large $\R^4$'s in Stein surfaces},
Algebr.\ Geom.\ Topol.\ {\bf 16} (2016), 1637--1681.

\bibitem[BG]{BG}
\v Z.\ Bi\v zaca and R.\ Gompf,
{\em Elliptic surfaces and some simple exotic $\R^4$'s},
J.\ Differential Geom.\
{\bf 43} (1996), 458--504.

\bibitem[Br]{Br}
M.\ Brown,
{\em A proof of the generalized Schoenflies theorem},
Bull.\ Amer.\ Math.\ Soc.\ {\bf66} (1960), 74--76.

\bibitem[CG]{CG}
J.\ Calcut and R.\ Gompf,
{\em On uniqueness of end sums and 1--handles at infinity},
in preparation.

\bibitem[C]{C}
A.\ Casson,
{\em Three lectures on new infinite constructions in 4--dimensional manifolds},
(notes prepared by L.~Guillou),
A la Recherche de la Topologie Perdue,
Progress in Mathematics
vol.\ 62,
Birkh\"auser, 1986, pp.~201--244.

\bibitem[CE]{CE}
K.\ Cieliebak and Y.\ Eliashberg,
{\em From Stein to Weinstein and Back -- Symplectic Geometry of Affine Complex Manifolds},
Colloquium Publications {\bf 59}, Amer.\ Math.\ Soc., Providence, 2012.

\bibitem[CFHS]{CFHS}
C.\ Curtis, M.\ Freedman, W.\ Hsiang and R.\ Stong,
{\em A decomposition theorem for h-cobordant smooth simply connected compact 4-manifolds},
Invent.\ Math.\ {\bf123} (1996), 343--348.

\bibitem[DF]{DF}
S.\ DeMichelis and M.\ Freedman,
{\em Uncountably many exotic $\R^4$'s in standard 4--space},
J.\ Differential Geom.\ {\bf 35} (1992), 219--254.

\bibitem[D1]{D}
S.\ Donaldson,
{\em An application of gauge theory to four dimensional topology},
J.\ Differential Geom.\
{\bf 18} (1983), 279--315.

\bibitem[D2]{Dpi1}
S.\ Donaldson,
{\em The orientation of Yang-Mills moduli spaces and 4-manifold topology},
J.\ Differential Geom.\
{\bf 26} (1987), 397--428.

\bibitem[D3]{Dpoly}
S.\ Donaldson,
{\em Polynomial invariants for smooth four-manifolds},
Topology
{\bf 29} (1990), 257--315.

\bibitem[FS1]{FSpi1}
R.\ Fintushel and R.\ Stern,
{\em $\SO(3)$-connections and the topology of 4-manifolds},
J.\ Differential Geom.\
{\bf 20} (1984), 523--539.

\bibitem[FS2]{FS}
R.\ Fintushel and R.\ Stern,
{\em Knots, links and 4-manifolds},
Invent.\ Math.\ {\bf134} (1998), 363--400.

\bibitem[FS3]{FS2}
R.\ Fintushel and R.\ Stern,
{\em Double node neighborhoods and families of simply connected 4-manifolds with $b^+=1$},
J.\ Amer.\ Math.\ Soc.\ {\bf19} (2006), 171--180.

\bibitem[F]{F}
M.\ Freedman,
{\em The topology of four-dimensional manifolds},
J.\ Diff.\ Geom.\ {\bf 17} (1982),  357--453.

\bibitem[FQ]{FQ}
M.\ Freedman and F.\ Quinn,
{\em Topology of 4-Manifolds},
Princeton Math.\ Ser., {\bf 39},
Princeton Univ.\ Press,
Princeton, NJ, 1990.

\bibitem[FT]{FT}
M.\ Freedman and L.\ Taylor,
{\em A universal smoothing of four-space},
J.\ Diff.\ Geom.\ {\bf 24} (1986),  69--78.

\bibitem[Fu]{Fu}
M.\ Furuta,
{\em Monopole equation and the $\frac{11}{8}$-Conjecture},
Math.\ Res.\ Lett.\ {\bf 8} (2001), No.\ 3,  279--291.

\bibitem[G1]{3R4}
R.\ Gompf,
{\em Three exotic $\R^4$'s and other anomalies},
J.\ Differential Geom.\ {\bf 18} (1983),  317--328.

\bibitem[G2]{infR4}
R.\ Gompf,
{\em An infinite set of exotic $\R^4$'s},
J.\ Differential Geom.\ {\bf 21} (1985),  283--300.

\bibitem[G3]{menag}
R.\ Gompf,
{\em An exotic menagerie},
J.\ Differential Geom.\ {\bf 37} (1993),  199--223.

\bibitem[G4]{symp}
R.\ Gompf,
{\em A new construction of symplectic manifolds},
Ann.\ of Math.\ (2) {\bf142} (1995), 527--595.

\bibitem[G5]{Ann}
R.\ Gompf,
{\em Handlebody construction of Stein surfaces},
Ann.\ of Math.\ (2) {\bf148} (1998), 619--693.

\bibitem[G6]{steindiff}
R.\ Gompf,
{\em Smooth embeddings with Stein surface images},
J.\ of Topology {\bf 6} (2013), 915--944.

\bibitem[G7]{MinGen}
R.\ Gompf,
{\em Minimal genera of open 4-manifolds}, 
Geom.\ Topol.\ {\bf 21} (2017), 107--155.

\bibitem[G8]{InfCork}
R.\ Gompf,
{\em Infinite order corks},
arXiv:1603.05090 ,
Geom.\ Topol.\ to appear.

\bibitem[G9]{Cork2}
R.\ Gompf,
{\em Infinite order corks via handle diagrams}, 
arXiv:1607.04354 ,
Algebr.\ Geom.\ Topol.\ to appear.

\bibitem[GS]{GS}
R.\ Gompf and A.\ Stipsicz,
{\em 4-Manifolds and Kirby Calculus},
Grad.\ Studies in Math.\ {\bf 20},
Amer.\ Math.\ Soc., Providence, 1999.

\bibitem[KS]{KS}
R.\ Kirby and L.\ Siebenmann,
{\em Foundational Essays on Triangulations and Smoothings of Topological Manifolds},
Ann.\ Math.\ Studies 88, Princeton University Press, 1977.

\bibitem[M]{M}
R.\ Matveyev,
{\em A decomposition of smooth simply-connected h-cobordant 4-manifolds},
J.\ Differential Geom.\ {\bf 44} (1996),  571--582.

\bibitem[Ma1]{Ma}
B.\ Mazur,
{\em On embeddings of spheres},
Bull.\ Amer.\ Math.\ Soc.\ {\bf65} (1959), 59--65.

\bibitem[Ma2]{Mcork}
B.\ Mazur,
{\em A note on some contractible 4-manifolds},
Ann.\ of Math.\ (2) {\bf73} (1961), 221--228.

\bibitem[Mo]{Mo}
M.\ Morse,
{\em A reduction of the Schoenflies extension problem},
Bull.\ Amer.\ Math.\ Soc.\ {\bf66} (1960), 113--115.

\bibitem[Q]{Q}
F.\ Quinn,
{\em Ends of maps. III: dimensions 4 and 5},
J.\ Differential Geom.\ {\bf 17} (1982),  503--521.

\bibitem[S]{S}
N.\  Sunukjian,
{\em A note on knot surgery},
J.\ Knot Theory Ramifications {\bf 24} No. 9 (2015), 1520003  (5 pages).

\bibitem[Tn1]{T}
M.\ Tange,
{\em Finite order corks},
arXiv:1601.07589 .

\bibitem[Tn2]{Tn2}
M.\ Tange,
{\em Non--existence theorems on infinite order corks},
arXiv:1609.04344 .

\bibitem[Tn3]{Tn3}
M.\ Tange,
{\em Notes on Gompf's infinite order cork},
arXiv:1609.04345 .

\bibitem[Tb]{Tb}
C.\ Taubes,
{\em Gauge theory on asymptotically periodic 4--manifolds},
J.\ Differential Geom.\ {\bf 25} (1987),  363--430.

\bibitem[T1]{Ta}
L.\ Taylor,
{\em An invariant of smooth 4--manifolds},
Geom.\ Topol.\ {\bf1} (1997), 71--89.

\bibitem[T2]{Ta2}
L.\ Taylor,
{\em Smooth Euclidean 4--spaces with few symmetries},
Geom.\ Topol.\ Monogr.\ {\bf 2} (1999),  563--569.

\bibitem[Z]{Z}
B.\ Zimmermann,
{\em On the classification of finite groups acting on homology 3-spheres},
Pacific J.\ Math.\  {\bf 217} (2004),  387--395.

\end{thebibliography}
\end{document}